\documentclass[]{interact}

\usepackage{wrapfig}
\usepackage{epstopdf}

\usepackage[longnamesfirst,sort]{natbib}
\bibpunct[, ]{(}{)}{;}{a}{,}{,}

\theoremstyle{plain}
\newtheorem{theorem}{Theorem}[section]
\newtheorem{lemma}[theorem]{Lemma}

\theoremstyle{definition}

\newtheorem{example}[theorem]{Example}

\theoremstyle{remark}
\newtheorem{remark}{Remark}

\begin{document}


\title{ Method of Quasidifferential Descent in the Problem of \\ Bringing a Nonsmooth System from One Point to Another }

\author{
\name{A.~V. Fominyh \textsuperscript{a}\thanks{The work was supported by the Russian Science Foundation (project no. 21-71-00021)}}
\affil{\textsuperscript{a}St. Petersburg State University, St. Petersburg, Russia.}
}

\maketitle

\begin{abstract}
The paper considers the problem of constructing program control for an object described by a system with a quasidifferentiable right-hand side. The control aim is to bring the system from a given initial position to a given final state in given finite time. The admissible controls are piecewise continuous vector-functions with values from a parallelepiped. The original problem is reduced to unconditional minimization of a functional. Herewith, the new technical idea is implemented to consider phase trajectory and its derivative as independent variables (and to take the natural relation between them into account via a special penalty function). This idea qualitatively simplified the quasidifferential structure and allowed to overcome the principal difficulties in constructing the steepest descent direction. The quasidifferentiability of the functional is proved, necessary conditions for its minimum are obtained in terms of quasidifferential. In contrast to the existing ones, due to the mentioned idea to ``separate'' the trajectory and its derivative the obtained optimality conditions in the paper are pointwise which may be effectively checked at discrete time moments. In order to solve the obtained minimization problem in the functional space the quasidifferential descent method is applied. Then the discretization is implemented. In contrast to majority of existing methods when the initial problem is discretized, here the discretization is implemented after the quasidifferential is already obtained. It is proved
that in order to construct the quasidifferential descent direction one has
to find the Hausdorff deviation of one convex compact (minus superdifferential)
from another convex compact (the subdifferential) at each time moment of the
discretization and to implement the corresponding interpolation. The quasidifferential descent directions are calculated independently at each time moment of discretization due to the comparatively simple quasidifferential structure, possible to obtain via the technical idea noted. Herewith, the functional is constructed in such a way that it is possible to check whether the resulting stationary point is its global minimum point. The algorithm developed is demonstrated by examples. The proposed method can be applied to nonsmooth optimal control problem in Lagrange form (additionally the integral with a quasidifferentiable integrand is to be minimized). The method advantages and disadvantages are discussed in details.
\end{abstract}

\bigskip

\section{Introduction}
Despite the rich arsenal of methods accumulated over the more than 60-year history of the development of optimal control theory, most of them deal with classical systems whose right-hand sides are continuously differentiable functions of their arguments. There are approaches that do not require these smoothness conditions on systems. However, they usually use direct discretization or some kind of ``smoothing'' process; both of these approaches lead to losing some of the information about  ``behavior'' of the system as well as to finite-dimensional problems of huge dimensions. The paper presented is aimed at solving the problem of bringing a nonsmooth (but only quasidifferentiable) system from one point to another. The relevance of considering such systems is due to their ability to more accurately and more fully describe the ``behavior'' of an object in many cases.

In order to solve the problem of this paper, we will use a combination of reducing the original problem to the problem of minimizing a functional in some functional space as well as the apparatus of quasidifferential calculus. The concept of ``quasidifferential'' was introduced by V. F. Demyanov. A rich and constructive calculus has been developed for this nonsmooth optimization object (see \cite{demrub}). In a finite-dimensional case quasidifferntiable functions are those whose directional derivative may be represented as a sum of the maximum of the scalar product of the direction and the vector from a convex compact set (called a subdifferential) and of the minimum of the scalar product of the direction and the vector from a convex compact (called a superdifferential). The pair of a subdifferential and a superdifferential is called a quasidifferential. (See the strict definition of quasidifferential applied to the right-hand sides of the system below). The class of quasidifferentiable functions is wide. In particular, it includes all the functions that can be represented as a superposition of the finite number of maxima and minima of continuously differentiable functions. The concept of quasidifferential was generalized onto functional spaces in the works of M. V. Dolgopolik (see, e. g.  \cite{dolgopolik1}, \cite{Dolgopolikcodiff}).

Let us make a brief review of some papers on nonsmooth control problems. Such works as \cite{Vinter1}, \cite{Vinter2}, \cite{Frankowska}, \cite{Mordukhovich}, \cite{Ioffe}, \cite{Shvartsman} are devoted to classical necessary optimality conditions in the form of the maximum principle for nonsmooth control problems in various formulations (including the case of the presence of phase constraints). In paper \cite{ito} the minimum conditions in the form of Karush---Kuhn---Tucker are obtained  for nonsmooth problems of mathematical programming in a general problem statement with applications to nonsmooth problems of optimal control. In paper \cite{olivera} on the basis of ``maximum principle invexity'' some sufficient conditions are obtained for nonsmooth control problems. The author of the paper presented also constructed some theoretical results in the problem of program control in systems whose right hand-sides contain modules of linear functions; 
in this problem the necessary minimum conditions are obtained in terms of quasidifferentials (see \cite{Fominyh2}). For the first time, quasidifferential (in the finite-dimensional case) was used to study nonsmooth control problems in work \cite{demshablnik}. The works listed are mainly of theoretical nature; and it is difficult to apply rather complex minimum conditions obtained there to specific control problems with systems with nonsmooth right-hand sides. Let us mention some works devoted directly to the construction of numerical methods for solving a problem similar to that considered in this paper. The author of this paper used the methods of subdifferential and hypodifferential descents earlier to construct optimal control with a subdifferentiable cost functional and the system with a continuously differentiable right-hand side (\cite{Fominyh1}) as well as to solve the problem of bringing a continuously differentiable system from one point to another (in paper \cite{Fominyh3}). In work \cite{Fominyh4} a finite-dimensional version of the quasidifferential descent method was applied to the optimal control problem in Mayer form with smooth right-hand side of a system and with a nondifferentiable objective functional. In work \cite{outrata} the optimal control problem is considered which assumes, roughly speaking, the subdifferentiability of an objective functional and continuous differentiability of right-hand sides of a system. The approach of this work is based on minimization of the discretized augmented cost functional via bundle methods. In papers \cite{Gorelik1}, \cite{Gorelik2} the minimax control problem is considered; with the help of a specially constructed smooth penalty function it is reduced to a classical continuously differentiable problem. In work \cite{Morzhin} a subdifferentiable penalty function is constructed in order to take the constraints on control into account; after that the subdifferential smoothing process is also used. In papers \cite{Mayne1}, \cite{Mayne2} the exact penalty function constructed (in order to take phase constraints into account) is also subdifferentiable; an algorithm for minimizing the derivative of this function with respect to the direction is considered. After ``transition'' to continuously differentiable problems a widely developed arsenal for solving classical control problems can be used in order to solve them. In works \cite{Skandari1}, \cite{Skandari2} more general problems are considered in which the nonsmoothness of right-hand sides of the system describing a controlled object is allowed. Here with the help of basis functions (Fourier series) the smoothing of system right-hand sides is also carried out, after which the Chebyshev pseudospectral method is used to solve the problem. In paper \cite{RossFahroo} the direct method for optimizing trajectories of nonsmooth optimal control problems is proposed based on the Legendre's pseudospectral method.

The method considered in this paper belongs to the so called direct methods of the variational analysis (see \cite{demtam}). The method is also ``continuous'' unlike most methods in literature as it is not based on direct discretization of the original problem. Although similar methods have already been applied to some problems of variational calculus and optimal control, so far it was impossible to apply this  method to nonsmooth control problems. The main difficulty was in a too complicated form of  quasidifferentials and optimality conditions obtained. The new technical idea of the current paper is to consider phase trajectory and its derivative as independent variables (and to take the natural relation between these variables into account via penalty function of a special form). To the best of the author's knowledge, this idea is used in literature for the first time. It allowed to simplify the quasidifferential structure of the functional under consideration and to solve the problem of finding the steepest descent direction of the minimized functional. 

\section{Basic Definitions and Notation}
The paper is organized as follows. Section 2 contains the notation of the paper as well as definitions of the quasidifferentials of the functions (the functionals) required. In Section~3 the problem statement and the main assumptions are presented. In Section~4 the original problem is reduced to the unconstrained minimization one. The quasidifferentiability of the main functional is proved in Section 5; after that minimum conditions for the unconstrained problem are obtained. The quasidifferential descent method is described in Section 6. In this section we also discuss the methods for solving the auxiliary problems arising during the basic algorithm implementation. Justification of possibility of finding the steepest descent direction at discrete time moments is carried out in the section as well. Section 7 contains the numerical examples illustrating the method realization with a rather detailed analysis of the problems considered. In Section 8 advantages and disadvantages of the method are discussed. Section~9 summarizes the main results of the paper. 

In the paper we will use both quasidifferentials of functions in a finite-dimensional space and quasidifferentials of functionals
in a functional space. Despite the fact that the second concept generalizes the first one, for convenience we introduce
definitions for both of these cases and for those specific functions (functionals) and their variables and spaces which are
considered.

Consider the space $\mathbb R^n \times \mathbb R^m$ with the standard norm. Let $g = [g_1, g_2] \in \mathbb R^n \times \mathbb R^m$ be an arbitrary vector. Suppose that at every time moment $t \in [0, T]$ at the point $(x, u)$ there exist such convex compact sets $\underline \partial f_i(x,u,t)$, $\overline \partial f_i(x,u,t)$ $\subset \mathbb R^n \times \mathbb R^m$, $i = \overline {1, n}$, that 
\begin{equation}
\label{5.3} 
\frac{\partial f_i(x,u,t)}{\partial g} = \lim_{\alpha \downarrow 0} \frac{1}{\alpha} \Big(f_i(x+\alpha g_1, u + \alpha g_2, t) - f_i(x,u,t)\Big) = \end{equation}
$$
= \max_{v \in \underline \partial f_i(x,u,t)} \langle v, g \rangle + \min_{w \in \overline \partial f_i(x,u,t)} \langle w, g \rangle, \quad i = \overline {1, n}. 
$$

In this case the function $f_i(x,u,t)$, $i = \overline{1,n}$, is called quasidifferentiable at the point $(x, u)$ and the pair $\mathcal{D}f_i(x, u) = \big[ \underline \partial f_i(x,u,t),  \overline \partial f_i(x,u,t) \big]$ is called a quasidifferential of the function $f_i(x,u,t)$ (herewith, the sets $\underline \partial f_i(x,u,t)$ and $\overline \partial f_i(x,u,t)$ are called a subdifferential and a superdifferential respectively of the function $f_i(x,u,t)$ at the point $(x,u)$). 

From expression (\ref{5.3}) one can see that at each $t \in [0, T]$ the following formula holds true:
\begin{equation}
\label{5.3''} f_i(x + \alpha g_1, u + \alpha g_2, t) = f_i(x, u, t) + \alpha  \frac{\partial f_i(x,u,t)}{\partial g} + o_i(\alpha, x, u, g, t), \end{equation}
$$\quad  \frac{o_i(\alpha, x, u, g, t)}{\alpha} \rightarrow 0, \ \alpha \downarrow 0, \quad i = \overline {1, n}.$$

If for each number $\varepsilon > 0$ there exist such numbers $\delta > 0$ and $\alpha_0 > 0$ that at $\overline g \in B_{\delta}(g)$ and at $\alpha \in (0, \alpha_0)$ one has $| o_i(\alpha, x, u, \overline g, t) | < \alpha \varepsilon $, $i = \overline {1, n}$, then the function $f_i(x,u,t)$, $i = \overline {1, n}$, is called uniformly quasidifferentiable at the point $(x, u)$. Note \cite{demvas} that if at each $t \in [0, T]$ the function $f_i(x,u,t)$, $i = \overline {1, n}$, is quasidifferentiable at the point $(x, u)$ and is locally Lipschitz continuous in the vicinity of the point $(x, u)$, then it is uniformly quasidifferentiable at the point~$(x, u)$. If for the uniformly quasidifferentiable function $f_i(x,u,t)$, $i = \overline {1, n}$, in expression (\ref{5.3''}) one has $\displaystyle{\frac{o_i(\alpha, x, u, g, t)}{\alpha} \rightarrow 0}$, $\alpha \downarrow 0$, $i = \overline {1, n}$, uniformly in $t \in [0, T]$, then such a function is called absolutely uniformly quasidifferentiable.
\newline

Consider the space $C_n[0, T] \times P_n[0, T] \times P_m[0, T]$ with the following norm: \linebreak $L^2_n [0, T] \times L^2_n [0, T] \times L^2_m [0, T]$. Let $g = [g_1, g_2, g_3] \in C_n[0, T] \times P_n[0, T] \times P_m[0, T]$ be an arbitrary vector-function. Suppose that at the point $(x, z, u)$ exist such convex weakly* compact sets $\underline \partial {I(x, z, u)}, \overline \partial {I(x, z, u)} \subset \big( C_n[0, T] \times P_n[0, T] \times P_m[0, T],$ $|| \cdot ||_{L^2_n [0, T] \times L^2_n [0, T] \times L^2_m [0, T]} \big) ^*$ that  
\begin{equation}
\label{5.3'} 
\frac{\partial I(x, z, u)}{\partial g} = \lim_{\alpha \downarrow 0} \frac{1}{\alpha} \Big(I(x+\alpha g_1, z+\alpha g_2, u + \alpha g_3) - I(x, z, u)\Big) =  \end{equation}
$$
= \max_{v \in \underline \partial I(x, z, u)} v(g) + \min_{w \in \overline \partial I(x, z, u)} w(g).
$$

In this case the functional $I(x,z,u)$ is called quasidifferentiable at the point $(x, z, u)$ and the pair $\mathcal{D}I(x, z, u) = \big[ \underline \partial {I(x, z, u)},  \overline \partial {I(x, z, u)} \big]$ is called a quasidifferential of the functional $I(x, z, u)$ (herewith, the sets $\underline \partial {I(x, z, u)}$ and $\overline \partial {I(x, z, u)}$ are called a \linebreak subdifferential and a superdifferential respectively of the functional $I(x,z,u)$ at the point $(x,z,u)$). 

From expression (\ref{5.3'}) one can see that the following formula holds true:
\begin{equation}
\label{5.6''} 
I(x + \alpha g_1, z + \alpha g_2, u + \alpha g_3) = I(x, z, u) + \alpha  \frac{\partial I(x,z,u)}{\partial g} + o(\alpha, x, z, u, g),
\end{equation} 
$$ \quad  \frac{o(\alpha, x, z, u, g)}{\alpha} \rightarrow 0, \ \alpha \downarrow 0.$$

\section{Statement of the Problem}
Consider the system of ordinary differential equations
\begin{equation}
\label{5.1}
\dot{x}(t) = f(x(t),u(t),t) 
\end{equation}
with the initial point
\begin{equation}
\label{5.2}
x(0) = x_{0}.
\end{equation}

In formula (\ref{5.1}) $f(x,u,t)$, $t~\in~[0, T]$, is a given $n$-dimensional vector-function; $T > 0$ is a known finite time moment. In formula (\ref{5.2}) $x_0 \in \mathbb R^n$ is a given vector. 

The $n$-dimensional vector-function $x(t)$ of phase coordinates is assumed to be piecewise continuously differentiable on $[0, T]$. The $m$-dimensional vector-function $u(t)$ of controls is supposed to be piecewise continuous on $[0, T]$. The vector-function $f(x,u,t)$ is supposed to be continuous on its domain; and each of its components $f_i(x, u, t)$, $i = \overline {1, n}$, --- to be quasidifferentiable and locally Lipschitz continuous in the pair $(x, u)$ of variables at $t \in [0, T]$. 

Under the assumptions made for system (\ref{5.1}), (\ref{5.2}), the classical solution existence and uniqueness theorems hold true, at least, in some neighborhood of the initial point. 

As noted above, we assume that each trajectory $x(t)$ is a piecewise continuously differentiable vector-function and $u(t)$ is a piecewise continuous vector-function on~$[0, T]$. If $t_0 \in [0, T)$ is a discontinuity point of the vector-function $u(t)$, then for definiteness we assume that $u(t_0) = \lim \limits_{t \downarrow t_0} u(t)$. At the point $T$ put $u(T) = \lim \limits_{t \uparrow T} u(t)$. With the assumptions and the notations made we can suppose that the vector-function $x$ belongs to the space $C_n [0, T]$, the vector-function~$\dot x$ belongs to the space $P_n [0, T]$ and the vector-function $u$ belongs to the space $P_m [0, T]$.

Introduce the set of admissible controls
\begin{equation} 
\label{5.4''} 
U  = \Big\{ u \in P_m[0, T] \ \big| \ \underline{u}_i \leqslant u_i(t) \leqslant \overline{u}_i, \ i = \overline{1, m}, \ t \in [0, T] \Big\}.
\end{equation}
Here $\underline{u}_i, \overline{u}_i \in \mathbb R, i = \overline{1, m}$, are given numbers.

{\bf Constrained Control Problem.} It is required to find such a control $u^{*} \in U$ that brings the corresponding (in the sense of equation (\ref{5.1})) trajectory~$x^* \in C_n[0, T]$ from initial point~(\ref{5.2}) to the final state 
\begin{equation} 
\label{5.5}
x(T) = x_{T},
\end{equation}
where $x_T \in \mathbb R^n$ is a given vector. 

We suppose that there exists such a control $u^* \in U$ (and the corresponding trajectory $x^* \in C_n[0,T]$). 

\section{Reduction to a Variational Problem}
The aim of this section is to reduce the {\bf Constrained Control Problem} stated above to {\bf Unconstrained Variational Problem} below.
Construct the functional taking into account different constraints on the object and on control which are given in the statement of the problem. Let $z(t) = \dot x (t)$ (under the assumptions made, $z \in P_n[0,T]$), then according to (\ref{5.2}) (where the initial state of the system is given) we have \begin{equation} 
\label{5.6} \displaystyle{x(t) = x_0 + \int_0^t z(\tau) d \tau}. \end{equation}
 
Construct the following functional on the space $P_n[0, T] \times P_m[0, T]$:
$$ \mathcal I(z, u) = \sum_{i=1}^n \int_0^T \Big| z_i (t) - f_i\Big(x_0 + \int_0^t z(\tau) d \tau, u(t), t\Big) \Big | dt + \frac{1}{2} \Big( x_0 + \int_0^T z(t) dt - x_T \Big)^2 + $$
$$ + \sum_{i=1}^m \int_0^T \max \Big\{\underline{u}_i - u_i(t), 0 \Big\} dt + \sum_{i=1}^m \int_0^T \max \Big\{{u}_i(t) - \overline{u}_i(t), 0 \Big\} dt. $$
In the functional $\mathcal I(z, u)$ the first summand (which is a sum) takes into account differential constraint (\ref{5.1}), the second summand takes into account constraint (\ref{5.5}) on the final state of the system, the third summand (consisting of two sums) takes into account constraint (\ref{5.4''}) on control. Note that this functional is nonnegative for any of its arguments and $\mathcal I (z^*, u^*)  = 0$ iff the pair $(x^*, u^*) \in C_n[0, T] \times P_m[0, T]$ is a solution of the original problem, i. e. the control $u^*$ belongs to the set $U$ of admissible controls and brings the corresponding trajectory $\displaystyle{x^*(t) = x_0 + \int_0^t z^*(\tau) d \tau}$ from the given initial position $x_0$ to the given final state $x_T$ in the time~$T$.

Transition to the ``space of derivatives'' $\left( z \in P_n[0, T] \right)$ has been used in many works of V.~F.~Demyanov and his students to study various variational and control problems. Under some natural additional assumptions one can
prove the quasidifferentiability of the functional~$\mathcal I (z, u)$ in the space $P_n[0, T] \times P_m[0, T]$ as a normed space with the norm $L^2_n [0, T] \times L^2_m [0, T]$. However, the quasidifferential of this functional has a rather complicated structure which makes it practically unsuitable for constructing numerical methods. Therefore, it is proposed to consider some
modification of this functional, ``forcibly'' considering the points $z$ and $x$ to be ``independent'' variables. Since, in fact, there is relationship (\ref{5.6}) between these variables (which naturally means that the vector-function $z(t)$ is a derivative of the vector-function $x(t)$), let us take it into account by adding the corresponding (last) term when constructing the new functional on the space
 $C_n[0, T] \times P_n[0, T] \times P_m[0, T]$:   
$$
I(x, z, u) = I_1(x, z, u) + I_2(z) + I_3(u) + I_4(x, z) = $$
\begin{equation} 
\label{5.7}
= \sum_{i=1}^n \int_0^T \Big| z_i (t) - f_i(x(t), u(t), t) \Big| dt + \frac{1}{2} \Big( x_0 + \int_0^T z(t) dt - x_T \Big)^2 + 
\end{equation}
$$ + \sum_{i=1}^m \int_0^T \max \Big\{\underline{u}_i - u_i(t), 0 \Big\} dt + \sum_{i=1}^m \int_0^T \max \Big\{{u}_i(t) - \overline{u}_i, 0 \Big\} dt + $$ $$ + \frac{1}{2} \int_0^T \Big( x(t) - x_0 - \int_0^t z(\tau) d \tau \Big)^2 dt. $$
Note that this functional is also nonnegative for any of its arguments and $ I (x^*, z^*, u^*)  = 0$ iff the pair $(x^*, u^*) \in C_n[0, T] \times P_m[0, T]$ is a solution of the original problem, i. e. the control $u^*$ belongs to the set $U$ of admissible controls and brings the corresponding trajectory $\displaystyle{x^*(t) = x_0 + \int_0^t z^*(\tau) d \tau}$ from the given initial position $x_0$ to the given final state $x_T$ in the time $T$. It is obvious that if some of the right endpoint coordinates of an object are free, then we put the corresponding summands of the functional $I_2(z)$ equal to zero. It is also obvious that if some of the restrictions on controls are absent, one has to remove the corresponding summands from the functional~$I_3(u)$. In both these cases we keep for the functional $I(x,z,u)$ its notation.  

Despite the fact that the dimension of functional $I(x, z, u)$ arguments is $n$ more the dimension of functional $\mathcal I(z, u)$ arguments, the structure of its quasidifferential (in the space $C_n[0, T] \times P_n[0, T] \times P_m[0, T]$ as a normed space with the norm \linebreak $L^2_n [0, T] \times L^2_n [0, T] \times L^2_m [0, T]$), as will be seen from what follows, is much simpler than the structure of the functional $\mathcal I(z, u)$ quasidifferential. This will allow us to construct
a numerical method for solving the original problem. 
\newline
\newline

{\bf Unconstrained Variational Problem.}
So the initial problem has been reduced to finding an unconstrained global minimum point $(x^*, z^*, u^*)$ of the functional $I(x, z, u)$ on the space $$X = \Big( C_n[0, T] \times P_n[0, T] \times P_m[0, T], || \cdot ||_{L^2_n [0, T] \times L^2_n [0, T] \times L^2_m [0, T]} \Big).$$

\begin{remark} \label{rm5.1}
Note the following fact. Since, as is known, the space $\Big( C_n[0, T], || \cdot ||_{L^2_n [0, T]} \Big)$ is everywhere dense in the space ${L^2_n [0, T]}$ and the space $\Big( P_n[0, T], || \cdot ||_{L^2_n [0, T]} \Big)$ is also everywhere dense in the space ${L^2_n [0, T]}$, then the space $X^*$ conjugate to the space $X$ introduced in the previous paragraph is isometrically isomorphic \cite{kolfom} to the space ${L^2_n [0, T] \times L^2_n [0, T] \times L^2_m [0, T]}$.
\end{remark}

\section{Necessary Minimum Conditions }
Let us formulate a minimum condition for the functional $I(x, z, u)$ that follows from its construction. Recall that the functional $I(x, z, u)$ is defined on the space \linebreak $C_n[0, T] \times P_n[0, T] \times P_m[0, T]$.

\begin{theorem} \label{th_5.4.1}
Let Assumptions 5.2.1, 5.2.2 be satisfied. In order for the point $(x^*, z^*, u^*)$ to minimize the functional $I(x, z, u)$, it is necessary and sufficient to have $I(x^*, z^*, u^*) =~0$.
\end{theorem}

In order to obtain a more constructive (than that given in Theorem \ref{th_5.4.1}) minimum condition useful for constructing
numerical methods for solving the problem posed, preliminarily, let us investigate the differential properties of the functional $I(x, z, u)$.

Using the classical variation one can directly prove the Gateaux differentiability of the functional $I_2(z)$; we have
$$
 \nabla I_2(z) = x_0 + \int_0^T z(t) dt - x_T.  
 $$
 By quasidifferential calculus rules \cite{DolgNormed} one may put 
 $$ \mathcal{D} \, I_2(z) = \Big[ \underline \partial I_2(z), \, \overline \partial I_2(z) \Big] := \bigg[ x_0 + \int_0^T z(t) dt - x_T, \ 0_n \bigg]. $$
 
 Formally denote $\displaystyle{ \underline \partial \varphi_2(x, z, u, t) = \left(0_n, x_0 + \int_0^T z(t) dt - x_T, 0_m \right)'}$, \linebreak $\overline \partial \varphi_2(x, z, u, t) = \left(0_n, 0_n, 0_m \right)'$.
 \newline
 
 Using the classical variation and integration by parts one can directly check (cf. \cite{demtam}) the Gateaux differentiability of the functional $I_4(x, z)$; we obtain
 $$\nabla I_4(x, z) = \begin{pmatrix}
\displaystyle{ x - x_0 - \int_0^t z(\tau) d\tau } \\
\displaystyle{ -\int_t^T \Big( x(\tau) - x_0 - \int_0^\tau z(s) ds \Big) d\tau  }
\end{pmatrix}. 
$$   
 
  By quasidifferential calculus rules \cite{DolgNormed} one may put 
 $$ \mathcal{D} \, I_4(x, z) = \Big[ \underline \partial I_4(x, z), \, \overline \partial I_4(x, z) \Big] := \left[ \begin{pmatrix}
\displaystyle{ x - x_0 - \int_0^t z(\tau) d\tau } \\
\displaystyle{ -\int_t^T \Big( x(\tau) - x_0 - \int_0^\tau z(s) ds \Big) d\tau  }
\end{pmatrix}, \ \begin{pmatrix}
0_n \\
0_n  
\end{pmatrix} \right ]. $$

 Formally denote $\overline \partial \varphi_4(x, z, u, t) = \left(0_n, 0_n, 0_m \right)'$,  $\displaystyle{ \underline \partial \varphi_4(x, z, u, t)=}$ $\displaystyle{  = \left(x(t) - x_0 - \int_0^t z(\tau) d\tau, -\int_t^T \Big( x(\tau) - x_0 - \int_0^\tau z(s) ds \Big) d\tau, 0_m \right)'}$.  
 \newline
 
 Study now the differential properties of the functionals $I_1(x, z, u)$ and $I_3(u)$. For this, we prove the following theorem for a functional of a more general form. 
 
\begin{theorem} \label{th5.4.2}
 Let the functional
 $$
 J(\xi) = \int_0^T \varphi(\xi(t), t) dt
 $$
be given where $\xi \in P_l[0, T]$, the function $\varphi(\xi, t)$ is continuous and is also absolutely uniformly quasidifferentiable, with the quasidifferential $[\underline \partial \varphi(\xi, t), \overline \partial \varphi(\xi, t)]$. Suppose also that the mappings $t \rightarrow \underline \partial \varphi(\xi(t), t)$ and $t \rightarrow \overline \partial \varphi(\xi(t), t) $ are upper semicontinuous.  
\newline

Then the functional $J(\xi)$ is quasidifferentiable, i. e. 

1) The derivative of the functional $J(\xi)$ in the direction $g \in P_l[0,T]$ exists and is of the form
\begin{equation}
\label{5.10} 
\frac{\partial J(\xi)}{\partial g} = \lim_{\alpha \downarrow 0} \frac{1}{\alpha} \big(J(\xi+\alpha g) - J(\xi)\big) = \max_{v \in \underline \partial J(\xi)} v(g) + \min_{w \in \overline \partial J(\xi)} w(g),
   \end{equation}
and the sets $\underline \partial J(\xi)$, $\overline \partial J(\xi)$ are of the form
\begin{equation}
\label{5.12} \underline \partial J(\xi) = \Big\{ v \in \left(P_l[0, T], || \cdot ||_{L_l^2} \right)^* \ \big| \  v(g) = \int_0^T \langle \upsilon(t), g(t) \rangle dt \quad \forall g \in P_l[0,T], \end{equation}
$$ \upsilon \in L_l^\infty[0, T], \quad \upsilon(t) \in \underline \partial \varphi(\xi(t), t) \ \forall t \in [0, T] \Big\}.$$
\newline

  \begin{equation}
\label{5.13}  \overline \partial J(\xi) = \Big\{ w \in \left(P_l[0, T], || \cdot ||_{L_l^2} \right)^* \ \big| \  w(g) = \int_0^T \langle \varpi(t), g(t) \rangle dt \quad \forall g \in P_l[0,T], \end{equation}
$$ \varpi \in L_l^\infty[0, T], \quad \varpi(t) \in \overline \partial \varphi(\xi(t), t) \ \forall t \in [0, T] \Big\}.$$
 
 2) The sets $\underline \partial J(\xi)$, $\overline \partial J(\xi)$ are convex and weakly* compact subsets of the space \linebreak $\big( P_l[0, T], || \cdot ||_{L^2_l [0, T]} \big)^*$. 
 \end{theorem}

%
%
\begin{proof}
Prove statement 1). 

Insofar as the function $\varphi(\xi, t)$ is quasidifferentiable by assumption, then for every $g \in P_l[0,T]$ and for each $\alpha > 0$ we have (see formula (\ref{5.3''}))
\begin{equation}
\label{5.16} 
J(\xi+\alpha g) - J(\xi) =  \int_0^T \max_{v \in \underline \partial \varphi(\xi,t)} \langle v(t), \alpha g(t) \rangle dt + \int_0^T  \min_{w \in \overline \partial \varphi(\xi,t)} \langle w(t), \alpha g(t) \rangle dt +
 \end{equation} 
  $$
 + \int_0^T {o(\alpha, \xi(t), g(t), t)} dt, \quad \frac{o(\alpha, \xi(t), g(t), t)}{\alpha} \rightarrow 0, \ \alpha \downarrow 0, \ t \in [0,T]. 
 $$

At this point let us check that the integrals in the right-hand side of this formula are correctly defined. 

Insofar as $\xi, g \in P_l[0,T]$ and the function $\varphi(\xi, t)$ is continuous, then for each $\alpha > 0$ the functions  $t \rightarrow \varphi(\xi(t), t)$ and $t \rightarrow \varphi(\xi(t) + \alpha g(t), t)$ belong to the space~$L^\infty_1 [0,T]$.

Under the assumption made, the mappings $t \rightarrow \underline \partial \varphi(\xi(t), t)$ and $t \rightarrow \overline \partial \varphi(\xi(t), t)$ are upper semicontinuous and then are also measurable \cite{filblag}. Then due to the piecewise continuity and the boundedness of the function $g(t)$ and due to continuity of the scalar product we obtain that for each $\alpha > 0$ the mappings $t \rightarrow \max_{v(t) \in \underline \partial \varphi(\xi(t),t)} \langle v(t), \alpha g(t) \rangle$ and $t \rightarrow \min_{w(t) \in \overline \partial \varphi(\xi(t),t)} \langle w(t), \alpha g(t) \rangle$ are upper semicontinuous \cite{AubinFrankowska} and then are also measurable~\cite{filblag}. During the statement 2) proof it will be shown that under the assumptions made, the sets $\underline \partial \varphi(\xi,t)$ and $\overline \partial \varphi(\xi,t)$ are bounded uniformly in $t \in [0, T]$, from here taking into account the fact that $g \in P_l[0,T]$, check that for each $\alpha > 0$ the mappings $t \rightarrow \max_{v(t) \in \underline \partial \varphi(\xi(t),t)} \langle v(t), \alpha g(t) \rangle$ and $t \rightarrow \min_{w(t) \in \overline \partial \varphi(\xi(t),t)} \langle w(t), \alpha g(t) \rangle$ are also bounded uniformly in $t \in [0, T]$. Indeed, fix some $g \in P_l[0,T]$ and $\alpha > 0$ and for each $t \in [0,T]$ take such a vector $\overline v(t) \in \underline \partial \varphi(\xi(t),t)$ that $\langle \overline v(t), \alpha g(t) \rangle = \max_{v(t) \in \underline \partial \varphi(\xi(t),t)} \langle v(t), \alpha g(t) \rangle$ (the vector~$\overline v(t)$ exists since for each $t \in [0,T]$ the set $\underline \partial \varphi(\xi(t),t)$ is a convex compact). Then by Cauchy-Schwarz inequality $\langle \overline v(t), \alpha g(t) \rangle \leqslant \alpha ||\overline v(t)||_{\mathbb R^l} ||g(t)||_{\mathbb R^l} $, and the value on the right-hand side is bounded (uniformly in $t \in [0, T]$) since $g \in P_l[0,T]$ and since the set $\underline \partial \varphi(\xi(t),t)$ is bounded uniformly in $t \in [0, T]$. (The justification regarding the mapping $t \rightarrow \min_{w(t) \in \overline \partial \varphi(\xi(t),t)} \langle w(t), \alpha g(t) \rangle$ is carried out in a completely analogous fashion.)  So we finally have that for each $\alpha > 0$ the mappings $t \rightarrow \max_{v(t) \in \underline \partial \varphi(\xi(t),t)} \langle v(t), \alpha g(t) \rangle$ and $t \rightarrow \min_{w(t) \in \overline \partial \varphi(\xi(t),t)} \langle w(t), \alpha g(t) \rangle$ belong to the space $L^\infty_1 [0,T]$. 
\newline
\newline

Then for every $\alpha > 0$ one has $t \rightarrow {o(\alpha, \xi(t), g(t), t)} \in  L^\infty_1 [0,T]$ and due to
the absolutely uniformly quasidifferentiability of the function $\varphi(\xi, t)$ we have 
\begin{equation}
\label{5.11}  \frac{o(\alpha, \xi(t), g(t), t)}{\alpha} =: \frac{o(\alpha)}{\alpha} \rightarrow 0, \ \alpha \downarrow 0. \end{equation}

Now our aim is to ``bring the operations of taking maximum and minimum out of the integral'', i. e. to obtain the expression in the right-hand side of formula (\ref{5.10}).
 
Consider the functional $\displaystyle{\int_0^T \max_{v \in \underline \partial \varphi(\xi,t)} \langle v(t), \alpha g(t) \rangle dt}$ in detail. For simplicity here we identify the vector-functions $v$ and $\upsilon$. For each $\alpha > 0$ and for each $t \in [0, T]$ we have the obvious inequality
$$
\max_{v \in \underline \partial \varphi(\xi,t)} \langle v(t), \alpha g(t) \rangle \geqslant \langle v(t), \alpha g(t) \rangle
$$
where $v(t)$ is a measurable selector of the mapping $t \rightarrow \underline \partial \varphi(\xi(t),t)$ (due to
the noted boundedness property of the set $\underline \partial \varphi(\xi, t)$ uniformly in $t \in [0, T]$ we have $v \in L^\infty_l [0,T]$) and by virtue of formula~(\ref{5.12}) form for every $\alpha > 0$ one has the inequality 
$$
\int_0^T \max_{v \in \underline \partial \varphi(\xi,t)} \langle v(t), \alpha g(t) \rangle dt \geqslant \max_{v \in \underline \partial J(\xi)} \int_0^T \langle v(t), \alpha g(t) \rangle dt. 
$$
Insofar as for each $\alpha > 0$ and for each $t \in [0, T]$ we have 
$$
\max_{v \in \underline \partial \varphi(\xi,t)} \langle v(t), \alpha g(t) \rangle \in \Big\{ \langle v(t), \alpha g(t) \rangle \ \big| \ v(t) \in \underline \partial \varphi(\xi(t),t) \Big\}
$$
and the set $\underline \partial \varphi(\xi,t)$ is closed and bounded at each fixed $t \in [0,T]$ by the definition of subdifferential and the mapping $t \rightarrow \underline \partial \varphi(\xi(t),t)$ is upper semicontinuous by assumption and also as the scalar product is continuous and $g \in P_l[0,T]$, then due to Filippov lemma \cite{Filippov} there exists a measurable selector $\overline{v}(t)$ of the mapping $t \rightarrow \underline \partial \varphi(\xi(t),t)$ that for each $\alpha > 0$ and for each $t \in [0, T]$ one has 
$$
\max_{v \in \underline \partial \varphi(\xi,t)} \langle v(t), \alpha g(t) \rangle = \langle \overline{v}(t), \alpha g(t) \rangle,
$$
so the element $\overline{v} \in \underline \partial J(\xi)$ brings the equality
in the previous inequality. So finally we obtain
\begin{equation}
\label{5.14} 
\int_0^T \max_{v \in \underline \partial \varphi(\xi,t)} \langle v(t), \alpha g(t) \rangle dt = \max_{v \in \underline \partial J(\xi)} \int_0^T \langle v(t), \alpha g(t) \rangle dt. 
\end{equation}

Consideration of the functional $\displaystyle{\int_0^T \min_{w \in \overline \partial \varphi(\xi,t)} \langle w(t), \alpha g(t) \rangle dt}$ is carried out in a completely analogous fashion (here we identify $w$ and $\varpi$). Taking formula (\ref{5.13}) form into account we have
\begin{equation}
\label{5.15} 
\int_0^T  \min_{w \in \overline \partial \varphi(\xi,t)} \langle w(t), \alpha g(t) \rangle dt = \min_{w \in \overline \partial J(\xi)} \int_0^T \langle w(t), \alpha g(t) \rangle dt.
\end{equation}

From expressions (\ref{5.16}), (\ref{5.11}), (\ref{5.14}), (\ref{5.15}) follows formula (\ref{5.10}) (see expression  (\ref{5.6''})). 

Prove statement 2). 

The convexity of the sets $\underline \partial J(\xi)$ and $\overline \partial J(\xi)$ immediately
follows from the convexity at each fixed $t \in [0, T]$ of the sets $\underline \partial \varphi(\xi,t)$ and $\overline \partial \varphi(\xi,t)$ respectively.

Prove the boundedness of the set $\underline \partial \varphi(\xi,t)$ uniformly in $t \in [0, T]$. Due to the upper
semicontinuity of the mapping $t \rightarrow \underline \partial \varphi(\xi(t),t)$ at each $t \in [0, T]$ there
exists such a number~$\delta(t)$ that under the condition $|\overline{t} - t| < \delta(t)$ the inclusion $\underline \partial \varphi(\xi(\overline{t}),\overline{t}) \subset \linebreak \subset B_r(\underline \partial \varphi(\xi({t}),{t}))$ holds true at $\overline{t} \in [0, T]$ where $r$ is some fixed finite positive number. The intervals $D_{\delta(t)}(t)$, $t \in [0, T]$, form open
cover of the segment $[0, T]$, so by Heine-Borel lemma one can take a finite subcover from this cover. Hence, there exists such a number $\delta > 0$ that for every $t \in [0, T]$ the inclusion $\underline \partial \varphi(\xi(\overline{t}),\overline{t}) \subset B_r(\underline \partial \varphi(\xi({t}),{t}))$ holds true once $|\overline{t} - t| < \delta$ and $\overline{t} \in [0, T]$. This means that for the segment $[0, T]$ there exists a finite partition $t_1 = 0, t_2, \dots, t_{N-1}, t_N = T$ with the diameter~$\delta$ such that $\underline \partial \varphi(\xi,{t}) \subset \bigcup\limits_{i=1}^N B_r(\underline \partial \varphi(\xi({t_i}),{t_i}))$ for all $t \in [0, T]$. It remains to notice that the set $\bigcup\limits_{i=1}^N B_r(\underline \partial \varphi(\xi({t_i}),{t_i}))$ is bounded due to the compactness of the set $\underline \partial \varphi(\xi,{t})$ at each fixed $t \in [0, T]$. The boundedness of the set $\overline \partial \varphi(\xi,t)$ uniformly in $t \in [0, T]$ may be proved similarly.

The weak* compactness of the set $\underline \partial J(\xi)$ in the space $\left( P_l[0,T], ||\cdot||_{L^2_l[0,T]} \right)^*$ follows from its weak compactness (in this space) by virtue of these topologies definitions \cite{kolfom}. Prove the weak compactness of the set $\underline \partial J(\xi)$ in the space $\left( P_l[0,T], ||\cdot||_{L^2_l[0,T]} \right)^*$. Note that by virtue of Remark~\ref{rm5.1} it is sufficient to consider the set $\underline \partial J(\xi)$ image (under an isometric isomorphic mapping from $\left( P_l[0,T], ||\cdot||_{L^2_l[0,T]} \right)^*$ to $L^2_{l}[0, T]$) in the space $L^2_{l}[0, T]$. For simplicity denote this image by $\underline \partial J(\xi)$ as well. So our aim now is to prove the weak compactness of the set $\underline \partial J(\xi)$ in the space $L^2_l [0, T]$. The space $L^2_{l}[0, T]$ is reflexive \cite{DunfordSchwarz}, so the set there is weakly compact if and only if it is bounded in norm and weakly closed \cite{DunfordSchwarz} in this space. The boundedness of this set in norm has been proved in the previous paragraph. In the next paragraph we prove that this set is weakly closed. The similar reasoning is valid for the set $\overline \partial J(\xi)$.

Prove that the set $\underline \partial J(\xi)$ is weakly closed. As shown in statement 1) proof
and at the beginning of statement 2) proof, the set $\underline \partial J(\xi)$ is convex and its elements $v$ belong to the space $L^\infty_l [0, T]$. Then all the more the set $\underline \partial J(\xi)$ is a convex subset of the space $L^2_l [0, T]$. Let us prove that the set $\underline \partial J(\xi)$ is closed in the weak topology of the space $L^2_l [0, T]$. Let $\{v_n\}_{n=1}^{\infty}$ be the sequence of functions from the set $\underline \partial J(\xi)$ converging to the function $v^*$ in the
strong topology of the space $L^2_l [0, T]$. It is known \cite{Munroe} that this sequence has the subsequence $\{v_{n_k}\}_{n_k=1}^{\infty}$ converging pointwise to $v^{*}$ almost everywhere on $[0, T]$, i. e. there exists such a subset $T' \subset [0, T]$ having the measure $T$ that for every point $t \in T'$ we have $v_{n_k}(t) \in \underline \partial \varphi(\xi(t), t)$ and  $v_{n_k}(t)$ converges to $v^{*}(t)$, $n_k = 1, 2, \dots$. But the set $\underline \partial \varphi(\xi(t), t)$ is closed at each $t \in [0, T]$ by the definition of subdifferential, hence for every $t \in T'$ we have $v^{*}(t) \in \underline \partial \varphi(\xi(t), t)$. So the set $\underline \partial J(\xi)$ is closed in the strong topology of the space $L^2_l [0, T]$ but it is also convex, so it is also closed \cite{DunfordSchwarz} in the weak topology of the space $L^2_l [0, T]$. One can prove that the set $\overline \partial J(\xi)$ is weakly closed (in $L^2_{l}[0, T]$) in a similar way.

The theorem is proved. 
\end{proof}
\bigskip
\bigskip

\begin{remark}
{\mbox The assumption of the absolute uniform quasidifferentiability is made in order to simplify the presentation. Via a special form of the mean value theorem \cite{DolgConverg} for quasidifferentials one can show that this assumption is actually redundant. }
\end{remark}

Thus, as one can see from Theorem \ref{th5.4.2}, the quasidifferentials of the functionals $I_1(x, z, u)$ and $I_3(u)$ are completely defined by the quasidifferentials of their integrands (at each time moment $t \in [0, T]$). Below there is the detailed description of calculating the quasidifferentials required as well as the main quasidifferential calculus rules.  

With the help of quasidifferential calculus rules \cite{demrub} at each $i \in \{1, \dots, n\}$, $j \in \{1, \dots, m\}$ and at each $t \in [0, T]$ calculate the quasidifferentials below.
$$\mathcal{D} \, \Big |z_i - f_i(x, u, t) \Big| = \left[ \left( \begin{array}{c} 0 \\ \vdots \\  0 \\ 1 \\ 0 \\ \vdots \\ 0 \end{array} \right) - \overline \partial f_i(x, u, t), \ - \underline \partial f_i(x, u, t) \right], $$
if $z_i - f_i(x, u, t) > 0$. Here $1$ is on the $(n+i)$-th place. 

$$\mathcal{D} \, \Big |z_i - f_i(x, u, t) \Big| = \left[ \left( \begin{array}{c} 0 \\ \vdots \\  0 \\ -1 \\ 0 \\ \vdots \\ 0 \end{array} \right) + \underline \partial f_i(x, u, t), \ \overline \partial f_i(x, u, t) \right], $$
if $z_i - f_i(x, u, t) < 0$. Here $-1$ is on the $(n+i)$-th place. 

$$\mathcal{D} \, \Big |z_i - f_i(x, u, t) \Big| = $$
$$ =  \left[ \mathrm{co} \left\{ \left( \begin{array}{c} 0 \\ \vdots \\  0 \\ 1 \\ 0 \\ \vdots \\ 0 \end{array} \right) - 2 \, \overline \partial f_i(x, u, t), \, \left( \begin{array}{c} 0 \\ \vdots \\  0 \\ -1 \\ 0 \\ \vdots \\ 0 \end{array} \right) + 2 \, \underline \partial f_i(x, u, t) \right\}, \ - \underline \partial f_i(x, u, t) +  \overline \partial f_i(x, u, t) \right], $$
if $z_i - f_i(x, u, t) = 0$. Here $1$ and $-1$ are on the $(n+i)$-th place. 

Put $\Big[\underline \partial \varphi_{1}(x, z, u, t), \overline \partial \varphi_{1}(x, z, u, t)\Big] = \mathcal{D}\sum\limits_{i=1}^n \Big|z_i(t) - f_i(x(t), u(t), t) \Big|$.

$$\mathcal{D} \max \Big\{u_j - \overline u_j, 0 \Big\} = \left[ \left( \begin{array}{c} 0 \\ \vdots \\  0 \\ 1 \\ 0 \\ \vdots \\ 0 \end{array} \right), \ 0_m \right], $$
if $u_j - \overline{u}_j > 0$. Here $1$ is on the $j$-th place.  

$$\mathcal{D} \max \Big\{u_j - \overline u_j, 0 \Big\} = \left[ 0_m, \ 0_m \right], $$
if $u_j - \overline{u}_j < 0$. 

$$\mathcal{D} \max \Big\{u_j - \overline u_j, 0 \Big\} = \left[ \mathrm{co} \left\{ \left( \begin{array}{c} 0 \\ \vdots \\  0 \\ 1 \\ 0 \\ \vdots \\ 0 \end{array} \right), \, 0_m \right\}, \ 0_m \right], $$
if $u_j - \overline{u}_j = 0$. Here $1$ is on the $j$-th place.  

$$\mathcal{D} \max \Big\{\underline u_j - u_j, 0 \Big\} = \left[ \left( \begin{array}{c} 0 \\ \vdots \\  0 \\ -1 \\ 0 \\ \vdots \\ 0 \end{array} \right), \ 0_m \right], $$
if $\underline u_j - {u}_j(t) > 0$. Here $-1$ is on the $j$-th place.  

$$\mathcal{D} \max \Big\{\underline u_j - u_j, 0 \Big\} = \left[ 0_m, \ 0_m \right], $$
if $\underline u_j - {u}_j(t) < 0$. 

$$\mathcal{D} \max \Big\{\underline u_j - u_j, 0 \Big\} = \left[ \mathrm{co} \left\{ \left( \begin{array}{c} 0 \\ \vdots \\  0 \\ -1 \\ 0 \\ \vdots \\ 0 \end{array} \right), \, 0_m \right\}, \ 0_m \right], $$
if $\underline u_j - {u}_j(t) = 0$. Here $-1$ is on the $j$-th place.  

Put $\Big[\underline \partial \varphi_{3}(x, z, u, t), \overline \partial \varphi_{3}(x, z, u, t)\Big] = \mathcal{D}\left(\sum\limits_{i=1}^m \max \Big\{u_j(t) - \overline u_j, 0 \Big\} + \sum\limits_{j=1}^m \max \Big\{\underline u_j - u_j(t), 0 \Big\}\right)$.

In the previous paragraph formulas the subdifferentials $\underline \partial f_i(x, u, t)$ and the superdifferentials $\overline \partial f_i(x, u, t)$, $i = \overline {1, n}$, are calculated via quasidifferential calculus apparatus as well. Book \cite{demrub} contains a detailed description of these rules for a rich class of functions. Let us give just some of these rules which were used in the formulas of the previous paragraph. Let $\xi \in \mathbb R^l$. If the function~$\varphi(\xi)$ is quasidifferentiable at the point $\xi_0 \in \mathbb R^l$ and $\lambda$ is some number, then we have
$$ \lambda \, \mathcal{D} \varphi(\xi_0) = \Big[ \lambda \, \underline\partial \varphi(\xi_0), \lambda \, \overline\partial \varphi(\xi_0) \Big], \quad \text{if} \ \lambda \geqslant 0,$$
 $$ \lambda \, \mathcal{D} \varphi(\xi_0) = \Big[ \lambda \, \overline\partial \varphi(\xi_0), \lambda \, \underline\partial \varphi(\xi_0) \Big], \quad \text{if} \ \lambda < 0.$$
If the functions $\varphi_k(\xi)$, $k = \overline{1, r}$, are quasidifferentiable at the point $\xi_0 \in \mathbb R^l$, then the quasidifferntial of the function $\varphi(\xi) = \max_{k = \overline{1, r}} \varphi_k(\xi)$ at this point is calculated by the formula $$ \mathcal{D} \varphi(\xi_0) = \Big[\underline\partial \varphi(\xi_0) ,  \overline\partial \varphi(\xi_0) \Big],$$
$$\underline\partial \varphi(\xi_0) = \mathrm{co}\Bigg\{\underline\partial \varphi_k(\xi_0) - \sum_{i \in P(\xi_0), \ i \neq k} \overline\partial \varphi_i(\xi_0), \  k \in P(\xi_0) \Bigg\} ,$$
$$ \overline\partial \varphi(\xi_0) = \sum_{i \in P(\xi_0)} \overline\partial \varphi_i(\xi_0), $$
$$P(\xi_0) = \Big\{ k \in \{1, \dots, r\} \ \big| \ \varphi_k(\xi_0) = \varphi(\xi_0) \Big\} .$$
If the functions $\varphi_k(\xi)$, $k = \overline{1, r}$, are quasidifferentiable at the point $\xi_0 \in \mathbb R^l$, then the quasidifferntial of the function $\varphi(\xi) = \min_{k = \overline{1, r}} \varphi_k(\xi)$ at this point is calculated by the formula $$ \mathcal{D} \varphi(\xi_0) = \Big[\underline\partial \varphi(\xi_0) ,  \overline\partial \varphi(\xi_0) \Big],$$
$$\underline\partial \varphi(\xi_0) = \sum_{j \in Q(\xi_0)} \underline\partial \varphi_j(\xi_0) ,$$
$$ \overline\partial \varphi(\xi_0) = \mathrm{co}\Bigg\{\overline\partial \varphi_k(\xi_0) - \sum_{j \in Q(\xi_0), \ j \neq k} \underline\partial \varphi_j(\xi_0), \  k \in Q(\xi_0) \Bigg\} , $$
$$Q(\xi_0) = \Big\{ k \in \{1, \dots, r\} \ \big| \ \varphi_k(\xi_0) = \varphi(\xi_0) \Big\} .$$
Note also that if the function $\varphi(\xi)$ is subdifferentiable at the point $\xi_0 \in \mathbb R^l$, then its quasidifferential at this point  may be represented in the form $$\mathcal{D} \varphi(\xi_0) = \Big[ \underline\partial \varphi(\xi_0), 0_l \Big],$$ and if the function  $\varphi(\xi)$ is superdifferentiable at the point $\xi_0 \in \mathbb R^l$, then its quasidifferential at this point  may be represented in the form $$\mathcal{D} \varphi(\xi_0) = \Big[ 0_l, \overline\partial \varphi(\xi_0) \Big].$$ These two formulas can be taken as definitions of a subdifferentiable and a superdifferentiable function respectively. If the function $\varphi(\xi)$ is differentiable at the point $\xi_0 \in \mathbb R^l$, then its quasidifferential may be represented in the forms $$\mathcal{D} \varphi(\xi_0) = \Big[\varphi'(\xi_0), 0_l \Big] \ \mathrm{or} \ \mathcal{D} \varphi(\xi_0) = \Big[0_l, \varphi'(\xi_0) \Big], $$ where $\varphi'(\xi_0)$ is a gradient of the function $\varphi(\xi)$ at the point $\xi_0$. The latter fact indicates that there is not the only way to construct quasidifferential.
We also note that the subdifferential (the superdifferential) of the finite sum of quasidifferentiable functions is a sum of subdifferentials (superdifferentials) of the summands, i. e. if the functions $\varphi_k(\xi)$, $k = \overline{1, r}$, are quasidifferentiable at the point $\xi_0 \in \mathbb R^l$, then the quasidifferential of the function $\varphi(\xi) = \sum_{k=1}^r \varphi_k(\xi)$ at this point is calculated by the formula $$ \mathcal{D} \varphi(\xi_0) = \left[  \sum_{k=1}^r \underline\partial \varphi_k(\xi_0) ,  \sum_{k=1}^r \overline\partial \varphi_k(\xi_0) \right].$$

Via the rules given and formulas (\ref{5.12}) and (\ref{5.13}) we find the quasidifferentials $\mathcal{D} I_1(x, z, u)$ and $\mathcal{D} I_3(u)$.

We have the following final formula \cite{DolgNormed} for calculating the quasidifferential of the functional $I(x, z, u)$ at the point $(x, z, u)$
\begin{equation}
\label{5.16'}
\mathcal{D} \, I(x, z, u) = \Big[ \underline \partial I(x, z, u), \, \overline \partial I(x, z, u) \Big] = \left[ \sum_{k=1}^4 \underline \partial I_k(x, z, u), \ \sum_{k=1}^4 \overline \partial I_k(x, z, u) \right]
\end{equation} 
where formally $I_2(x, z, u) := I_2(z)$, $I_3(x, z, u) := I_3(u)$, $I_4(x, z, u) := I_4(x, z)$.

Let us formally denote $\underline \partial \varphi(\xi, t) = \sum\limits_{i=1}^4 \underline \partial \varphi_i(x, z, u, t)$, \linebreak $\overline \partial \varphi(\xi, t) = \sum\limits_{i=1}^4 \overline \partial \varphi_i(x, z, u, t)$.

Using the known minimum condition \cite{Dolgopolikcodiff} (of the functional $I(x, z, u)$ at the point $(x^*, z^{*}, u^{*})$ in this case) in terms of quasidifferential, we conclude that the following theorem holds true.

\begin{theorem} \label{th5.4.3}
Let Assumptions 5.2.1, 5.2.2 be satisfied. In order for the control $u^{*} \in U$ to bring system~(\ref{5.1}) from initial point (\ref{5.2}) to final state (\ref{5.5}) in the time $T$, it is necessary that for each measurable selection $w(\cdot)$ of the multivalued mapping $t \rightarrow \overline \partial \varphi(\xi^*, t)$ the following inclusion
\begin{equation}
\label{5.17}
- w(t) \in \underline \partial \varphi(\xi^*, t)
\end{equation}
holds true at almost each $t \in [0, T]$. 

If one has $I(x^*, z^{*}, u^{*}) = 0$, then condition (\ref{5.17}) is also sufficient.
\end{theorem}

\begin{remark}
Strictly speaking, minimum condition (\ref{5.17}) is formulated in the paper \cite{Dolgopolikcodiff} for a functional defined on another space, however, from the proof of that paper it is clear that this result remains valid for the case of the space considered in the paper.
\end{remark}

Theorem \ref{th5.4.3} already contains a constructive minimum condition since on its basis it is possible to construct the quasidifferential descent method; and for solving each of the subproblems arising during realization of this method (for a wide class of functions) there are known efficient algorithms for solving them. 

Once the steepest (the quasidifferential) descent direction has been constructed (see Section 6), one can apply some a numerical method (based on using this direction) of nonsmooth optimization in order to find stationary points of the functional $I(x, z, u)$. The steepest (quasidifferential) descent algorithm is used for numerical simulations of the paper.

\section{Quasidifferential Descent Method}

Let us describe the following quasidifferential descent method for finding stationary points of the functional $I(x, z, u)$. 

Fix an arbitrary initial point $(x_{(1)}, z_{(1)}, u_{(1)})$ from the space \linebreak $C_n[0, T] \times P_n[0, T] \times P_m[0, T]$. Let the point $(x_{(k)}, z_{(k)}, u_{(k)})$ from the space $C_n[0, T] \times P_n[0, T] \times P_m[0, T]$ be already constructed. If for each $t \in [0, T]$ minimum condition (\ref{5.17}) is satisfied (in practice, with some fixed accuracy $\overline{\varepsilon}$ in sense of $L^2$-norm \linebreak (see problem (\ref{5.19})) or at discrete time moments $t_i$, $i = \overline{1,N}$, with some fixed discretization rank $N$ in sense of $\mathbb R$-norm (see problem (\ref{5.20}))), then $(x_{(k)}, z_{(k)}, u_{(k)})$ is a stationary point of the functional $I(x, z, u)$ and the process terminates. Otherwise, put
$$
(x_{(k+1)}, z_{(k+1)}, u_{(k+1)}) = (x_{(k)}, z_{(k)}, u_{(k)}) + \gamma_{(k)} G(x_{(k)}, z_{(k)}, u_{(k)})
$$
where the vector-function $G(x_{(k)}, z_{(k)}, u_{(k)})$ is the quasidifferential descent direction of the functional $I(x, z, u)$ at the point $(x_{(k)}, z_{(k)}, u_{(k)})$ and the value $\gamma_{(k)}$ is a solution of the following one-dimensional problem
\begin{equation}
\label{5.18}
\min_{\gamma \geqslant 0} I \Big( (x_{(k)}, z_{(k)}, u_{(k)}) + \gamma G(x_{(k)}, z_{(k)}, u_{(k)}) \Big) = I \Big( (x_{(k)}, z_{(k)}, u_{(k)}) + \gamma_{(k)} G(x_{(k)}, z_{(k)}, u_{(k)})\Big). 
\end{equation}
In practice, the problem above is solved on some interval $[0, \overline{\gamma}]$ with some fixed $\overline \gamma$ value. Then $I (x_{(k+1)}, z_{(k+1)}, u_{(k+1)}) \leqslant I (x_{(k)}, z_{(k)}, u_{(k)}) $. 
If the sequence $ (x_{(k)}, z_{(k)}, u_{(k)}) $, \linebreak $k = 1,2, \dots$, is finite then its last point is a stationary point of the functional $I(x, z, u)$ by construction. If the sequence $(x_{(k)}, z_{(k)}, u_{(k)})$, $k = 1,2, \dots$, is infinite, then the described process may not leed to a stationary point of the functional $I(x, z, u)$ since the quasidifferential mapping $(x, z, u) \rightarrow \mathcal{D} I(x, z, u)$ is not continuous \cite{demrub} in Hausdorff metric.

As seen from the algorithm described, in order to realize the $k$-th iteration, one has to solve three subproblems. The first subproblem is to calculate the quasidifferential of the functional $I(x, z, u)$ at the point $(x_{(k)}, z_{(k)}, u_{(k)})$. With the help of quasidifferential calculus rules the solution of this subproblem is obtained in formula (\ref{5.16'}). The second subproblem is to find the quasidifferential descent direction $G(x_{(k)}, z_{(k)}, u_{(k)})$; the following two paragraphs are devoted to solving this subproblem. The third subproblem is one-dimensional minimization (\ref{5.18}); there are many effective methods \cite{Vasil'ev} to solve it. 

In order to obtain the vector-function $G(x_{(k)}, z_{(k)}, u_{(k)})$, consider the problem
\begin{equation}
\label{5.19}
\max_{w \in \overline{\partial} I(x_{(k)}, z_{(k)}, u_{(k)}) } \min_{v \in \underline{\partial} I(x_{(k)}, z_{(k)}, u_{(k)}) } \int_0^T \big( v(t) + w(t) \big )^2 dt. 
\end{equation}
Recall that we identify $v$ and $\upsilon$, $w$ and $\varpi$ for simplicity. Denote $\overline{v}$, $\overline{w}$ a solution of the problem above (below we will see that such a solution exists). (The vector-functions $\overline{v}$, $\overline{w}$, of course, depend on the point $(x_{(k)}, z_{(k)}, u_{(k)})$ but we omit this dependence in the notation for brevity.) Then the vector-function $G(x_{(k)}, z_{(k)}, u_{(k)}) = -\big( \overline{v} + \overline{w} \big)$ is a (not normed) quasidifferential descent direction \cite{demrub} of the functional $I(x, z, u)$ at the point $(x_{(k)}, z_{(k)}, u_{(k)})$. Note that the functional $I(x, z, u)$ quasidifferential at each time moment $t \in [0, T]$ is calculated independently (i. e. its quasidifferential, calculated at one time moment, does not depend on its quasidifferential, calculated at other time moment). 

\begin{remark}
Strictly speaking, in \cite{demrub} only a finite-dimensional space is considered, however, if we take into account the characterization \cite{balakrishnan} of the closest (in norm) element of a closed convex set in a Hilbert space through the nonnegativity of the corresponding scalar product, then the proof therein shows that the steepest descent direction formula remains true for the case considered in the paper.
\end{remark}


Now let us check that in order to solve problem (\ref{5.19}) in this case, one has to solve the problem 
\begin{equation}
\label{5.20}
\max_{w(t) \in \overline \partial \varphi(\xi_{(k)}, t)} \min_{v(t) \in \underline \partial \varphi(\xi_{(k)}, t) } \big( v(t) + w(t) \big )^2
\end{equation}
 for each $t \in [0, T]$. Note that this problem has a solution since we seek for the Hausdorff deviation of one convex compact from another (see below).
 
 Indeed, let $\overline v$, $\overline w \in L_n^\infty[0,T] \times L_n^\infty[0,T] \times L_m^\infty[0,T]$ (these vector-functions exist by Filippov lemma) be such that for each $t \in [0, T]$ we have 
 $$\big( \overline v(t) + \overline w(t) \big )^2 = \max_{w(t) \in \overline{\partial} \varphi(\xi_{(k)}, t) } \min_{v(t) \in \underline{\partial} \varphi(\xi_{(k)}, t) } \big( v(t) + w(t) \big )^2.$$
 Then we obtain
 $$\int_0^T \big( \overline v(t) + \overline w(t) \big )^2 dt = \int_0^T \max_{w(t) \in \overline{\partial} \varphi(\xi_{(k)}, t) } \min_{v(t) \in \underline{\partial} \varphi(\xi_{(k)}, t) } \big( v(t) + w(t) \big )^2 dt = $$
 $$ = \int_0^T \min_{v(t) \in \underline{\partial} \varphi(\xi_{(k)}, t) } \big( v(t) + \overline w(t) \big )^2 dt = \min_{v \in \underline{\partial} I(x_{(k)}, z_{(k)}, u_{(k)}) } \int_0^T \big( v(t) + \overline w(t) \big )^2 dt $$
 where the last equality holds due to Filippov lemma (cf. formula (\ref{5.15})). Hence 
 $$ \max_{w \in \overline{\partial} I(x_{(k)}, z_{(k)}, u_{(k)}) } \min_{v \in \underline{\partial} I(x_{(k)}, z_{(k)}, u_{(k)}) } \int_0^T \big( v(t) + w(t) \big )^2 dt \geqslant $$
 \begin{equation} \label{ineq1}
  \geqslant \int_0^T \max_{w(t) \in \overline{\partial} \varphi(\xi_{(k)}, t) } \min_{v(t) \in \underline{\partial} \varphi(\xi_{(k)}, t) } \big( v(t) + w(t) \big )^2 dt .
  \end{equation}
  
 Now fix some $\overline{\overline w} \in L_n^\infty[0,T] \times L_n^\infty[0,T] \times L_m^\infty[0,T]$. Again, by Filippov lemma we get
  $$ \min_{v \in \underline{\partial} I(x_{(k)}, z_{(k)}, u_{(k)}) } \int_0^T \big( v(t) + \overline{\overline w}(t) \big )^2 dt = \int_0^T \min_{v(t) \in \underline{\partial} \varphi(\xi_{(k)}, t) } \big( v(t) + \overline{\overline w}(t) \big )^2 dt \leqslant $$ 
  $$ \leqslant \int_0^T \max_{w(t) \in \overline{\partial} \varphi(\xi_{(k)}, t) } \min_{v(t) \in \underline{\partial} \varphi(\xi_{(k)}, t) } \big( v(t) + w(t) \big )^2 dt. $$
  Since the vector-function $\overline{\overline w}(t)$ was chosen arbitrarily, we obtain the inequality
$$ \max_{w \in \overline{\partial} I(x_{(k)}, z_{(k)}, u_{(k)}) } \min_{v \in \underline{\partial} I(x_{(k)}, z_{(k)}, u_{(k)}) } \int_0^T \big( v(t) + w(t) \big )^2 dt \leqslant $$
 \begin{equation} \label{ineq2}
  \leqslant \int_0^T \max_{w(t) \in \overline{\partial} \varphi(\xi_{(k)}, t) } \min_{v(t) \in \underline{\partial} \varphi(\xi_{(k)}, t) } \big( v(t) + w(t) \big )^2 dt. 
  \end{equation}
 From inequalities (\ref{ineq1}) and (\ref{ineq2}) we finally get the equality
 $$ \max_{w \in \overline{\partial} I(x_{(k)}, z_{(k)}, u_{(k)}) } \min_{v \in \underline{\partial} I(x_{(k)}, z_{(k)}, u_{(k)}) } \int_0^T \big( v(t) + w(t) \big )^2 dt = $$
 \begin{equation} \label{eq}
  = \int_0^T \max_{w(t) \in \overline{\partial} \varphi(\xi_{(k)}, t) } \min_{v(t) \in \underline{\partial} \varphi(\xi_{(k)}, t) } \big( v(t) + w(t) \big )^2 dt. 
  \end{equation}
  
  The equality (\ref{eq}) justifies that in order to solve problem (\ref{5.19}) it is sufficient to solve problem (\ref{5.20}) for each time moment $t \in [0, T]$. Once again we emphasize that this statement holds true due to the special structure of the quasidifferential which in turn takes place due to the ``separation'' implemented of the vector-functions \linebreak $x(t)$ and $\dot x(t)$ into ``independent'' variables.

Problem (\ref{5.20}) at each fixed $t \in [0, T]$ is a finite-dimensional problem of finding the Hausdorff deviation of one convex compact set (a minus superdifferential) from another convex compact set (a subdifferential). This problem may be effectively solved for a rich class of functions; its solution is described in the next paragraph. In practice, one makes a (uniform) partition of the interval $[0, T]$ and this problem is being solved for each point of the partition, i. e. one calculates $G((x_{(k)}, z_{(k)}, u_{(k)}), t_i)$ where \linebreak $t_i \in [0, T]$, $i = \overline{1, N}$, are discretization points (see notation of Lemmas \ref{lm5.5.1}, \ref{lm5.5.2} below). Under additional natural assumption Lemma \ref{lm5.5.1} below guarantees that the vector-function obtained via piecewise-linear interpolation of the quasidifferential descent directions calculated at each point of such a partition of the interval $[0, T]$ converges in the space $L^2_{2n + m}[0, T]$ (as the discretization rank $N$ tends to infinity) to the vector-function $G(x_{(k)}, z_{(k)}, u_{(k)})$ sought. Such an approximation guarantees that the following point $[x_{k+1}, z_{k+1}, u_{k+1}]$ ``does not leave'' the space \linebreak $C_n[0,T] \times P_n[0,T] \times P_m[0,T]$, providing the method correctness in this sense.

As noted in the previous paragraph, during the algorithm realization it is required to find the Hausdorff deviation of the minus superdifferential from the subdifferential of the functional $I (x, z, u)$ at each time moment of a (uniform) partition of the interval $[0, T]$. In this paragraph we describe in detail a solution (for a rich class of functions) of this subproblem for some fixed value $t \in [0, T]$. It is known \cite{demrub} that in many practical cases the subdifferential $\underline \partial \varphi(\xi, t)$ is a convex polyhedron $A(t) \subset \mathbb R^{2n+m}$ and analogously the superdifferential $\overline \partial \varphi(\xi, t)$ is a convex polyhedron $B(t) \subset \mathbb R^{2n + m}$. For example, if some function is a superposition of the finite number of maxima and minima of continuously differentiable functions, then its subdifferential and its superdifferential are convex polyhedra. Herewith, of course, the sets $A(t)$ and $B(t)$ depend on the point $(x, z, u)$. For simplicity, we omit this dependence in this paragraph notation. Find the Hausdorff deviation of the set $-B(t)$ from the set $A(t)$. It is clear that in this case it is sufficient to go over all the vertices $b_j(t)$, $j = \overline{1,s}$ (here $s$ is a number of vertices of the polyhedron $-B(t)$): find the Euclidean distance from every of these vertices to the polyhedron $A(t)$ and then among all the distances obtained choose the largest one. Let the Euclidean distance sought, corresponding to the vertex $b_j(t)$, $j = \overline{1,s}$, is achieved at the point $a_j(t) \in A(t)$ (which is the only one since $A(t)$ is a convex compact). Then the deviation sought is the value $||b_{\overline j}(t) - a_{\overline{j}}(t)||_{\mathbb R^{2n + m}}$, $\overline j \in \{1, \dots, s\}$. (Herewith, this deviation may be achieved at several vertices of the polyhedron $-B(t)$; in this case~$b_{\overline j}(t)$ denotes any of them.)
Note that the arising \linebreak \newline problem of finding the Euclidean distance from a point to a convex polyhedron can be effectively solved by various methods \cite{demmal}.


First, give a lemma with a simple condition which, on the one hand, is quite natural for applications and, on the other hand, guarantees that the function $\mathcal L(t)$ obtained via piecewise-linear interpolation of the sought function $\mathcal G \in L^{\infty}_1[0, T]$ converges to it in the space $L^2_1[0, T]$. 

\begin{lemma} \label{lm5.5.1}
 Let the function $\mathcal G \in L^{\infty}_1[0, T]$ satisfy the following condition: for every $\overline{\delta} > 0$ the function $\mathcal G(t)$ is piecewise continuous on the set $[0, T]$ with the exception of only the finite number of the intervals $\big(\overline t_1(\overline{\delta}), \overline t_2(\overline{\delta})\big), \dots, \big(\overline t_{r}(\overline{\delta}), \overline t_{r+1}(\overline{\delta})\big)$ whose union length does not exceed the number $\overline{\delta}$. 
 
 Choose a (uniform) finite splitting $t_1 = 0, t_2, \dots, t_{N-1}, t_N = T$ of the interval~$[0, T]$ and calculate the values $\mathcal G(t_i)$, $i = \overline{1, N}$, at these points. Let $\mathcal{L}(t)$ be the function obtained with the help of piecewise linear interpolation with the nodes $(t_i, \mathcal G(t_i))$, \linebreak $i = \overline{1, N}$. Then for each $\varepsilon > 0$ there exists such a number $\overline{N}(\varepsilon)$ that for every $N > \overline{N}(\varepsilon)$ one has $||\mathcal L - \mathcal G||^2_{L^2_1[0,T]} \leqslant \varepsilon$.
 \end{lemma}  

 \begin{proof} Denote $M(\overline{\delta}) := \bigcup\limits_{k=1}^{r} \big(\overline t_{k}(\overline{\delta}), \overline t_{k+1}(\overline{\delta})\big)$. We have 
  $$||\mathcal L - \mathcal G||^2_{L^2_1[0,T]} = \int_{M(\overline{\delta})} \big (\mathcal L(t) - \mathcal G(t) \big)^2 dt + \int_{[0, T] \setminus M(\overline{\delta})} \big (\mathcal L(t) - \mathcal G(t) \big)^2 dt.$$
   Fix the arbitrary number $\varepsilon > 0$. By lemma condition the function $\mathcal G(t)$ is bounded, the function $\mathcal L(t)$ is also bounded by construction for all (uniform) finite partitions of the interval $[0, T]$. Hence, there exists such $\overline\delta(\varepsilon)$ that the first summand does not exceed the value ${\varepsilon}/{2}$ for all (uniform) finite partitions
of the interval $[0, T]$. As assumed, the function $\mathcal G(t)$ is piecewise continuous and bounded on the set $[0, T] \setminus M(\overline{\delta}(\varepsilon))$, then there exists \cite{ryab} such a number $\overline{N}(\varepsilon)$ that for every (uniform) finite partition of the interval $[0, T]$ of the rank $N > \overline{N}(\varepsilon)$ the second summand (with such~$\overline{\delta}(\varepsilon))$ does not exceed the value ${\varepsilon}/{2}$. 

\end{proof} 
 
   Now give a lemma with a more general but less clear (compared to Lemma~\ref{lm5.5.1}) condition also guaranteeing that the function $\mathcal L(t)$ obtained via piecewise-linear interpolation of the sought function $\mathcal G \in L^{\infty}_1[0, T]$ converges to it in the space $L^2_1[0, T]$.
   
\begin{lemma} \label{lm5.5.2}
 Let the function $\mathcal G \in L^{\infty}[0, T]$ satisfy the following conditions: 
 \newline 1) for every $\overline{\varepsilon} > 0$ there exists such a closed set $T(\overline{\varepsilon}) \subset [0, T]$ that the function $\mathcal G(t) |_{T(\overline{\varepsilon})}$ is continuous and $|T'(\overline{\varepsilon})| < \overline{\varepsilon}$ where $T'(\overline{\varepsilon}) := [0, T] \setminus T(\overline{\varepsilon})$; 
 \newline
 2) for every $\delta > 0$ there exists such a number $\overline{\overline{N}}(\delta)$ that for each $N > \overline{\overline{N}}(\delta)$ we have $|M(\delta)| < \delta$ where $M(\delta) := \bigcup\limits_{k=2}^{N} [t_{k-1}, t_{k}]$; here we take the union of only such intervals $[t_{k-1}, t_{k}]$, in each of which at least one of the points $t_{k-1}$, $t_{k}$, $k \in \{2, \dots, N\}$, belongs to the set $T'(\overline{\varepsilon})$.
 
  Choose a (uniform) finite splitting $t_1 = 0, t_2, \dots, t_{N-1}, t_N = T$ of the interval~ $[0, T]$ and calculate the values $\mathcal G(t_i)$, $i = \overline{1, N}$, at these points. Let $\mathcal L(t)$ be the function obtained with the help of piecewise linear interpolation with the nodes $(t_i, \mathcal G(t_i))$, \linebreak $i = \overline{1, N}$. Then for each $\varepsilon > 0$ there exists such a number $\overline{N}(\varepsilon)$ that for every $N > \overline{N}(\varepsilon)$ one has $||\mathcal L - \mathcal G||^2_{L^2_1[0,T]} < \varepsilon$. 
  \end{lemma}
  
\begin{proof} Note that the first assumption of the lemma is always satisfied since it is nothing but formulation of Lusin theorem \cite{kolfom}. However, it is given in the lemma formulation since the set $T'(\varepsilon)$ introduced there is used in the second assumption of the lemma.
  
  Fix some number $\varepsilon > 0$. 
  
  Let $\mathcal P(t)$ be a ``polygonal extension'' of the function $\mathcal G(t) |_{T(\overline{\varepsilon})}$ onto the whole interval $[0, T]$ which may be constructed \cite{Cullum1969} due to the fact that the set $T(\overline{\varepsilon})$ is closed (see assumption 1) of the lemma). Then the function $\mathcal P(t)$ is continuous on $[0, T]$ and $\mathcal P(t) = \mathcal G(t)$ at $t \in T(\overline{\varepsilon})$. Herewith, one can check \cite{Cullum1969} that one may choose $\overline{\varepsilon}$ in such a way that  
  \begin{equation}
\label{5.22}
\displaystyle{\int_0^T (\mathcal P(t) - \mathcal G(t))^2 dt < \varepsilon/3}.
\end{equation} 
  
  Consider the expression \begin{equation}
\label{5.21} \displaystyle{\int_0^T \big (\mathcal L(t) - \mathcal P(t) \big)^2 dt} = \int_{M(\delta)} \big (\mathcal L(t) - \mathcal P(t) \big)^2 dt + \int_{[0, T] \setminus M(\delta)} \big (\mathcal L(t) - \mathcal P(t) \big)^2 dt. \end{equation}  
 
 Consider the first summand in the right-hand side of equality (\ref{5.21}). By construction the function $\mathcal P(t)$ is bounded, the function $\mathcal L(t)$ is also bounded by construction for all (uniform) finite partitions of the interval $[0, T]$. Then from assumption 2) of the lemma it follows that for $\varepsilon > 0$ there exists such $\delta(\varepsilon)$ that for every (uniform) partition of the interval $[0, T]$ of the rank $N > \overline{\overline{N}}(\delta(\varepsilon))$ one has
   \begin{equation}
\label{5.23} \int_{M(\delta)} \big (\mathcal L(t) - \mathcal P(t) \big)^2 dt < \varepsilon/3.
\end{equation} 

Consider the second summand in the right-hand side of equality (\ref{5.21}). Let $\overline {\mathcal L}(t)$ be a function obtained via piecewise-linear interpolation with the nodes $(t_i, \mathcal P(t_i))$, \linebreak $i = \overline{1, N}$. Insofar as the function $\mathcal P(t)$ is continuous on  $[0, T]$, then there exists \cite{ryab} such a number $\overline{\overline{\overline{N}}}(\varepsilon)$ that for every (uniform) partition of the interval $[0, T]$ of the rank $N > \overline{\overline{\overline{N}}}(\varepsilon)$ one has $\displaystyle{\int_0^T \big (\mathcal {\overline L}(t) - \mathcal P(t) \big)^2 dt < \varepsilon / 3}$. But at $t \in [0, T] \setminus M(\delta)$ we have $\mathcal L(t) = \mathcal {\overline L}(t)$ by construction (with the same rank of partitions involved in these functions construction), insofar as if $t_i \in [0, T] \setminus M(\delta)$, $i \in \{1, \dots, N\}$, then $t_i \in T(\varepsilon)$, and for such $t_i$ we have $\mathcal P(t_i) = \mathcal G(t_i)$. For every (uniform) partition of the interval $[0, T]$ of the rank $N > \overline{\overline{\overline{N}}}(\varepsilon)$ we then have
  \begin{equation}
\label{5.24} \int_{[0, T] \setminus M(\delta)} \big (\mathcal L(t) - \mathcal P(t) \big)^2 dt = \int_{[0, T] \setminus M(\delta)} \big (\overline {\mathcal L}(t) - \mathcal P(t) \big)^2 dt \leqslant \int_0^T \big (\overline {\mathcal L}(t) - \mathcal P(t) \big)^2 dt < \varepsilon / 3. 
\end{equation} 
  
  Take $\overline{N}(\varepsilon) = \max\Big\{\overline{\overline{N}}(\delta(\varepsilon)), \overline{\overline{\overline{N}}}(\varepsilon)\Big\}$. For every (uniform) partition of the interval $[0, T]$ of the rank $N > \overline{N}(\varepsilon)$ from  (\ref{5.22}), (\ref{5.23}), (\ref{5.24}) we finally have
    $$||\mathcal L - \mathcal G||^2_{L^2_1[0,T]} \leqslant \int_{0}^T \big (\mathcal L(t) - \mathcal P(t) \big)^2 dt + \int_{0}^T \big (\mathcal P(t) - \mathcal G(t) \big)^2 dt < \varepsilon / 3 + \varepsilon / 3 + \varepsilon / 3 = \varepsilon.$$
\end{proof} 

 \begin{remark} 
 The meaning of assumption~2) of Lemma \ref{lm5.5.2} is the requirement that $\mathcal G(t)$ does not have ``too many'' discontinuity points on the segment $[0, T]$. It may be directly verified that if the condition of Lemma \ref{lm5.5.1} on the function $\mathcal G(t)$ is fulfilled, then the condition required is satisfied. In the picture a simple example is given of a measurable bounded function with an infinite number of discontinuity points for which one may construct the function $\mathcal P(t)$ in such a way that the set $M(\delta)$ measure is arbitrarily small for a sufficiently large splitting rank. It is an example of an ``appropriate'' in the sense of Lemma \ref{lm5.5.2} assumption function. 
 \end{remark}
    \begin{wrapfigure}{r}{0.5\textwidth} \begin{center}
     \includegraphics[width=0.48\textwidth]{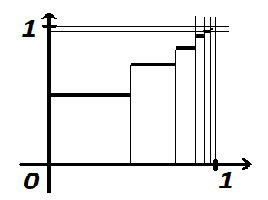}
  \end{center}    
      \end{wrapfigure}    
Let us give an example of the function for which this condition is violated. Let $\mathcal G(t)$ be the Dirichlet function on the segment $[0, 1]$, i. e. taking the value $1$ at rational points and taking the value $0$ at irrational points of this interval. If we take the function $\mathcal P(t) = 0$ $\forall t \in [0, 1]$ as a continuous one and as satisfying Lusin theorem applied to the function $\mathcal G(t)$, then condition~2) of Lemma \ref{lm5.5.2} will be violated with every rank of (uniform) splitting of the interval $[0, 1]$, insofar as with such a splitting all the splitting points will be rational, i. e. will belong to the set $T'(\overline\varepsilon)$ $\forall \overline \varepsilon > 0$, hence $|M(\delta)| = 1$ $\forall \delta > 0$ in this case. It is seen that in this example one has $||\mathcal L - \mathcal G||^2_{L_1^2[0,1]} =1$ for each function $\mathcal L(t)$ obtained via piecewise-linear interpolation of the function $\mathcal G(t)$ with a uniform splitting of the segment $[0, 1]$, insofar as with such a splitting we always have $\mathcal L(t) = 1$ $\forall t \in [0, 1]$. The Dirichlet function $\mathcal G(t)$ does not satisfy condition 2) of Lemma \ref{lm5.5.2}, insofar as this function has ``too many'' discontinuity points on the interval $[0, 1]$.  
\begin{remark}
Let us briefly note here the theoretical and practical significance of the constructed algorithm. Theoretically, the algorithm is interesting in that, in contrast to the currently known optimality conditions, the minimum conditions are obtained here in pointwise form. In practice, the algorithm is interesting in that the search for descent directions at individual sampling moments can be carried out independently of each other. These advantages are described in more detail in Section 8 (see also Remark \ref{rm_100}).
\end{remark}

\begin{remark} \label{rm5.5.2}
The main concepts described above can be applied to some more difficult control problems. For example, consider the problem of minimizing the functional $$J(x, z, u) = \int^T_0 f_0(x(t), \dot x(t), u(t), t) dt$$
under restrictions (\ref{5.1}), (\ref{5.2}), (\ref{5.5}) and $u \in U$ (see (\ref{5.4''})). 
The vector-function $f_0(x, \dot x,u,t)$ is supposed to be continuous on its domain and to be quasidifferentiable in $(x, z, u)$ and locally Lipschitz continuous at each point $(x, z, u)$ at each fixed $t \in [0, T]$.
 
Construct the functional
$$J(x, z, u) + \lambda I(x,z, u) $$
with the sufficiently large number $\lambda$. In practice it may be also more reasonably to take different penalty factors at the summands $I_i$, $i = \overline{1,4}$, of the functional $I(x, z, u)$ (see formula~(\ref{5.7})) corresponding to different restrictions. The aim of these penalty parameters choice is to satisfy the corresponding restrictions with an error required. Note that the functional constructed may be treated via the method proposed. 
\end{remark}

\section{Numerical Examples} 
Let us first explain the operation of the quasidifferential descent method using an illustrative example in which the iterations are given in detail. For simplicity of presentation this example is not chosen as a control problem but is a nonsmooth problem of the calculus of variations; however, according to the minimized functional structure, it fits the formulation of the problem considered.

\begin{example}
Consider the functional 
$$I(x,z) = \int_0^1 \Big|z(t) + |x(t)|\Big| dt + \int_0^1 |z(t)| dt + \int_0^T \Big( x(t) - x_0 - \int_0^t z(\tau) d \tau \Big)^2 dt  $$
with the initial point $x_0  = 0$ and with the obvious solution $x^*(t) = 0$, $z^*(t) = 0$ for all $t \in [0, 1]$. Note that the third summand here means that $z(t)$ must be the derivative of $x(t)$ (see the functional $I_4(x, z)$ in formula (\ref{5.7}) and formula (\ref{5.6})). 

Take the functions $x_{(1)}(t) = t-0.5$, $z_{(1)}(t) = 0$ as an initial approximation and discretize the segment $[0, 1]$ with a splitting rank, equal to $2$ (i. e. consider the points $0$, $0.5$, $1$ for further interpolation of the quasidifferential descent direction).  In this example let us use the following simplified notation $\left[ \underline{\partial} I(x_{(1)}, z_{(1)}), \overline{\partial} I(x_{(1)}, z_{(1)}) \right]$ in order to denote the corresponding quasidifferential (see formulas (\ref{5.12}) and (\ref{5.13})). \linebreak According to the algorithm, calculate the descent direction separately at these points. At the point $t=0$ we have $$\underline{\partial} I(x_{(1)}, z_{(1)}) = \mathrm{co} \left\{ \left( \begin{array}{c} 0 \\ -2 \end{array} \right), \left(\begin{array}{c} 2 \\ -2 \end{array} \right) \right\},$$  $$\overline{\partial} I(x_{(1)}, z_{(1)}) = \left( \begin{array}{c} 0 \\ 0 \end{array} \right).$$ Find the deviation of the set $-\overline{\partial}  I(x_{(1)}, z_{(1)})$ from the set $\underline{\partial} I(x_{(1)}, z_{(1)})$ at the point $t=0$ and obtain the quasidifferential descent direction $G(( x_{(1)},z_{(1)}), 0) = (0, 2)'$. At the point $t=0.5$ we have \bigskip \bigskip \bigskip \bigskip \bigskip $$
\underline{\partial} I(x_{(1)}, z_{(1)}) = \mathrm{co} \left\{ \left( \begin{array}{c} -0.25 \\ -2 \end{array} \right), \left(\begin{array}{c} 1.75 \\ -2 \end{array} \right),  \left( \begin{array}{c} -0.25 \\ 2 \end{array} \right), \left( \begin{array}{c} 1.75 \\ 2 \end{array} \right), \left( \begin{array}{c} -2.25 \\ 0 \end{array} \right), \left( \begin{array}{c} -0.25 \\ 0 \end{array} \right) \right\}, $$
$$\overline{\partial} I(x_{(1)}, z_{(1)}) = \mathrm{co} \left\{ \left( \begin{array}{c} 0 \\ -1 \end{array} \right), \left(\begin{array}{c} 0 \\ 1 \end{array} \right) \right\},$$ hence, one has $G((x_{(1)},z_{(1)}), 0.5) = (0, 0)'$. At the point $t=1$ we have $$\underline{\partial} I(x_{(1)}, z_{(1)}) = \mathrm{co} \left\{ \left( \begin{array}{c} 0 \\ 2 \end{array} \right), \left(\begin{array}{c} 2 \\ 2 \end{array} \right) \right\},$$  $$\overline{\partial} I(x_{(1)}, z_{(1)}) = \left( \begin{array}{c} 0 \\ 0 \end{array} \right),$$ so one has $G((x_{(1)},z_{(1)}), 1) = (0, -2)'$. By making the appropriate interpolation we obtain the quasidifferential descent direction of the functional $I(x,z)$ at the point $(x_{(1)}, z_{(1)})$, namely $G(x_{(1)}, z_{(1)}) = (0, -4t+2)'$. Construct the next point $(x_{(2)}(t), z_{(2)}(t)) =$ \linebreak $= (0 + \gamma_{(1)} \, 0, t-0.5 + \gamma_{(1)} (-4t+2))'$; having solved the one-dimensional minimization problem $\min_{\gamma \geqslant 0} I(x_{(2)}, z_{(2)})$, we have $\gamma_{(1)} = 0.25$, hence, $x_{(2)}(t) = 0 $, $z_{(2)}(t) = 0 $ for all $t \in [0, 1]$, i. e. in this case the quasidifferential descent method leads to the exact solution in one step. 

Of course, the initial approximation and discretization rank are chosen artificially here in order to demonstrate the essence of the method developed. If we take different initial approximation and discretization rank, then in a general case it will no longer be possible to obtain an exact solution in the finite number of steps. 
\end{example}




Consider now some examples of bringing a nonsmooth system from one point to another. In these problems the quasidifferential descent method led to an (approximate) minimum point of functional (\ref{5.7}). 

The calculations were performed in the package MatLab 18.0
on a computer with the 3.6 GHz AMD Ryzen 5 PRO 2400G CPU
and 8 GB of RAM. The solution of the one-dimensional minimization problem was carried out on the interval $[0, 1]$ (i. e. here \linebreak $\overline \gamma = 1$) in MatLab via incorporated algoruthm fminbnd() with its inner default parameters. In accordance with the documentation, in order to solve this problem in MatLab the golden mean method is used combined with parabolic interpolation \cite{Vasil'ev}. All the integrals were calculated in MatLab via the incorporated function evalf(int()) with its inner default parameters. In accordance with the documentation, in order to solve this problem in MatLab the Gauss-Kronrod method \cite{BerezinZhidkov} is used. In examples the parameter $\overline\varepsilon$ was ignored and the solution error was evaluated based on the value of the functional. In this case, the value at the right endpoint was calculated for the trajectories resulting from numerical integration via Euler forward scheme with the step $10^{-4}$ of the system (from the initial left endpoint given) with the control substituted which was obtained via the method. In order to solve the auxiliary problem of finding the Euclidean distance from a point to a convex polyhedron the first method of subsequent approximations of Malozemov-Demyanov-Mitchell algorithm was used with its parameter $\overline{\delta} = 10^{-4}$.




\begin{example}\label{ex_4.2}
Consider the system
$$\dot x_1(t) = -|x_1(t)|,$$
$$\dot x_2(t) = u(t).$$

It is required to find such a control $u^* \in U$ which brings this system from the initial point $x(0) = (0, 0)'$ to the final state $x(1) = (0, 0)'$ at the moment $T=1$. Herewith, put $\underline u = -1$, $\overline u = 1$, i.~e. we suppose that $-1 \leqslant u(t) \leqslant 1$ $\forall t \in [0, 1]$. Note that the boundary conditions are intentionally taken the same in order to have the solution $x_1(t) = 0$ for all $t \in [0, 1]$, so the function~$|x_1|$ is essentially nonsmooth at this point.   

The problem given is reduced to an unconstrained minimization of the functional 
$$
I(x,z,u) = \int_0^1 \Big| z_1 (t) + |x_1(t)| \Big | dt + \int_0^1 \Big| z_2 (t) - u(t) \Big | dt + $$ $$ +\frac{1}{2} \Big( \int_0^1 z_1(t) dt\Big)^2 + \frac{1}{2} \Big( \int_0^1 z_2(t) dt\Big)^2 + $$
$$ +  \int_0^1 \max \Big\{-1 - u(t), 0 \Big\} dt + \int_0^1 \max \Big\{{u}(t) - 1, 0 \Big\} dt + $$ 
$$ + \frac{1}{2} \int_0^1 \Big( x_1(t) - \int_0^t z_1(\tau) d \tau \Big)^2 dt + \frac{1}{2} \int_0^1 \Big( x_2(t) - \int_0^t z_2(\tau) d \tau \Big)^2 dt. $$

Take $(x_{(1)}, z_{(1)}, u_{(1)}) = (1, 1, 1, 1, 1)'$ as an initial point, then $I(x_{(1)}, z_{(1)}, u_{(1)})\approx \linebreak \approx 3.33333$. As the iteration number increased, the discretization rank gradually increased during the solution
of the auxiliary problem of finding the direction of the quasidifferential descent described in the algorithm and, in the end, the discretization step was equal to $10^{-1}$ (i. e. $N = 10$). At the $17$-th iteration the control $u_{(17)}$ was constructed:
\small
$$u(t) \approx 0.56234-1.61222 t-0.04788 t^3+0.78478 t^2,$$ 
\normalsize with the value of the functional $I(x_{(17)}, z_{(17)}, u_{(17)}) \approx 0.00462$, herewith \linebreak $x_1(T) \approx 0.00357$, $x_2(T) \approx 0.00538$. For the convenience, the Lagrange interpolation polynomial is given accurately approximating (that is, the interpolation error does not affect the value of the functional and the boundary values)
the resulting control.

Take $u_{(17)}$ as an approximation to the control $u^*$ sought. In order to verify the result obtained and to find the ``true'' trajectory, we substitute this control into the system given and integrate it numerically. As a result, we have the corresponding trajectory (which is an approximation to the one $x^*$ sought) with the values $x_1(T) = 0$, $x_2(T) = 0.00585$, so we see that the error on the right endpoint does not exceed the value $5 \times 10^{-3}$; the control restrictions are satisfied exactly. 

The computational time was 0 min 36 sec. The pictures illustrate the control and trajectories dynamics during the algorithm realization. 

 \begin{figure*}[h!]
\begin{minipage}[h]{0.3\linewidth}
\center{\includegraphics[width=1\linewidth]{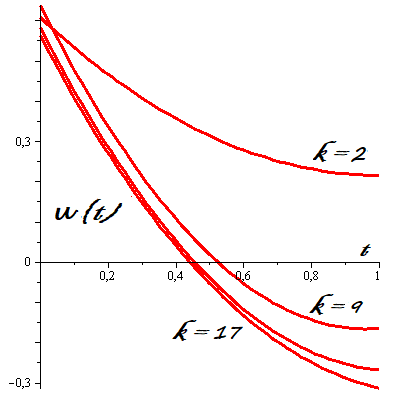} }
\end{minipage}
\hfill
\begin{minipage}[h]{0.3\linewidth}
\center{\includegraphics[width=1\linewidth]{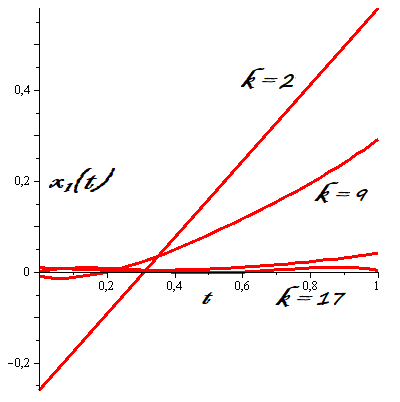} }
\end{minipage}
\hfill
\begin{minipage}[h]{0.3\linewidth}
\center{\includegraphics[width=1\linewidth]{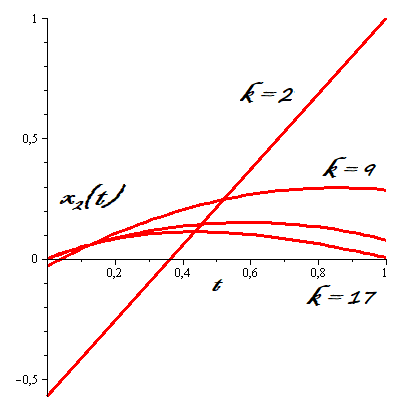} }
\end{minipage}
\caption{Example \ref{ex_4.2}, control and trajectories on iterations: 2, 9, 13, 17}
\end{figure*}

\end{example}

\begin{example} \label{ex_4.3}
It is required to bring the system
$$\dot x_1(t) = x_2(t) x_3(t) + u_1(t),$$
$$\dot x_2(t) = x_1(t) x_3(t) + u_2(t),$$
$$\dot x_3(t) = x_1(t) x_2(t) + u_3(t)$$
 from the initial point $x(0) = (1, 0, 0)'$ to the final state $x(1) = (0, 0, 0)'$ at the moment $T = 1$. Herewith, we suppose that the total control consumption is subject to the constraint
 $ \displaystyle{\int_0^1 |u_1(t)| + |u_2(t)| + |u_3(t)| dt = 1}.$
 This problem has a practical application to the optimal satellite stabilization and was considered in work \cite{krylov}. With the help of the new variable \linebreak $x_4(t) = \displaystyle{\int_0^t |u_1(\tau)| + |u_2(\tau)| + |u_3(\tau)| d \tau}$ reduce the problem given to problem \linebreak (\ref{5.1}), (\ref{5.2}), (\ref{5.5}) which is considered in the paper.
 
 Then we have the system
 $$\dot x_1(t) = x_2(t) x_3(t) + u_1(t),$$
$$\dot x_2(t) = x_1(t) x_3(t) + u_2(t),$$
$$\dot x_3(t) = x_1(t) x_2(t) + u_3(t),$$
$$\dot x_4(t) = |u_1(t)| + |u_2(t)| + |u_3(t)|,$$
with no restrictions on the control $u^* \in P_3[0,T]$ which is aimed at bringing the object from the initial point $x(0) = (1, 0, 0, 0)'$ to the final state $x(1) = (0, 0, 0, 1)'$ at the time moment $T = 1$. 

The problem given is reduced to an unconstrained minimization of the functional 
$$
I(x,z,u) = \int_0^1 \Big| z_1 (t) - x_2(t) x_3(t) - u_1(t) \Big| dt + \int_0^1 \Big| z_2 (t) - x_1(t) x_3(t) - u_2(t) \Big| dt + $$
$$+ \int_0^1 \Big| z_3 (t) - x_1(t) x_2(t) - u_3(t) \Big| dt + \int_0^1 \Big| z_4 (t) - |u_1(t)| - |u_2(t)| - |u_3(t)| \Big| dt +$$
$$ + \frac{1}{2} \Big( 1 + \int_0^1 z_1(t) dt\Big)^2 + \frac{1}{2} \Big( \int_0^1 z_2(t) dt\Big)^2 + \frac{1}{2} \Big( \int_0^1 z_3(t) dt\Big)^2 + \frac{1}{2} \Big( \int_0^1 z_4(t) dt - 1\Big)^2 +$$
$$ +\frac{1}{2} \sum_{i=1}^4 \int_0^1 \left( x_i(t) - x_{i}(0) - \int_0^t z_i(\tau) d \tau \right)^2 dt . $$

Take $(x_{(1)}, z_{(1)}, u_{(1)}) = (1+t, t, t, t, 1, 1, 1, 1, 0, 0, 0)'$ as an initial point, then $I(x_{(1)}, z_{(1)}, u_{(1)}) \approx 5.72678$. As the iteration number increased, the discretization rank gradually increased during the solution
of the auxiliary problem of finding the direction of quasidifferential descent described in the algorithm and, in the end, the discretization step was equal to $10^{-1}$ (i. e. $N=10$). At the $58$-th iteration the control $u_{(58)}$ was constructed: \small

$$u_1(t) \approx -2.88657 t^5+8.96619 t^4-9.30386 t^3+2.99867 t^2+0.10679 t-1.03399,$$ 
$$u_2(t) \approx -0.83764 t^5+0.85068 t^4+0.43135 t^3-0.54058 t^2+0.06933 t+0.00945,$$ 
$$u_3(t) \approx 0.13344 t - 0.01334, \ t \in [0, 0.1), \quad 0.03928 t^3-0.01848 t^2-0.00194 t, \ t \in [0.1, 0.6), $$
$$-0.94481 t^3+2.12646 t^2-1.55331 t+0.36987, \ t \in [0.6, 1],$$ 
\normalsize with the value of the functional $I(x_{(58)}, z_{(58)}, u_{(58)}) \approx 0.00551$, herewith, \linebreak $x_1(T) \approx 0.01251$, $x_2(T) \approx 0.00431$, $x_3(T) \approx 0.00431$, $x_4(T) \approx 1.0069$. For the convenience, the Lagrange interpolation polynomial is given accurately approximating (that is, the interpolation error does not affect the value of the functional and the boundary values)
the resulting control.

Take $u_{(58)}$ as an approximation to the control $u^*$ sought. In order to verify the result obtained and to find the ``true'' trajectory, we substitute this control into the system given and integrate it numerically. As a result, we have the corresponding trajectory (which is an approximation to the one $x^*$ sought) with the values $x_1(T) = 0.00514$, $x_2(T) = 0.00204$, $x_3(T) = 0.00051$, $x_4(T) = 1.00514$, so we see that the error on the right endpoint does not exceed the value $5 \times 10^{-3}$. 

The computational time was 1 min 43 sec. The pictures illustrate the control and trajectories dynamics during the algorithm realization. 

 \begin{figure*}[h!]
\begin{minipage}[h]{0.3\linewidth}
\center{\includegraphics[width=1\linewidth]{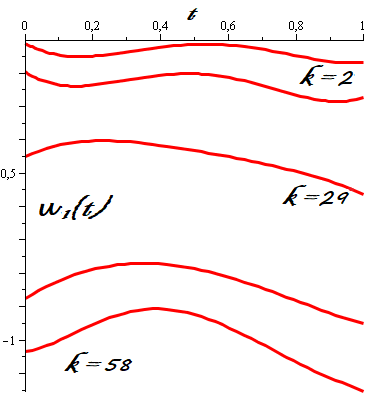} }
\end{minipage}
\hfill
\begin{minipage}[h]{0.3\linewidth}
\center{\includegraphics[width=1\linewidth]{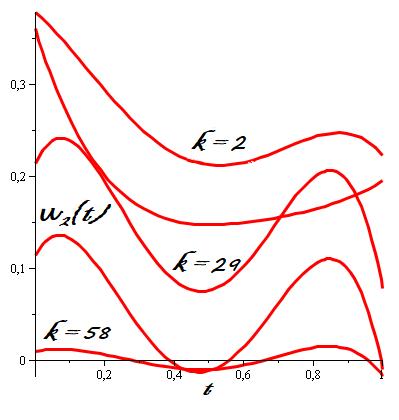} }
\end{minipage}
\hfill
\begin{minipage}[h]{0.3\linewidth}
\center{\includegraphics[width=1\linewidth]{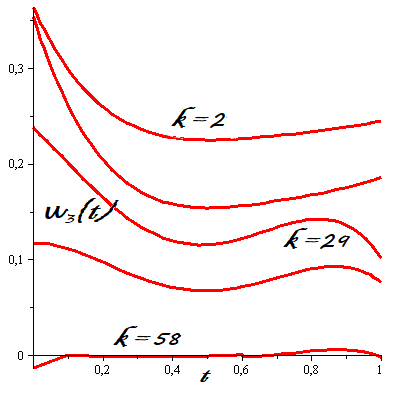} }
\end{minipage}
\caption{Example \ref{ex_4.3}, control on iterations: 2, 13, 29, 45, 58}
\end{figure*}

 \begin{figure*}[h!]
\begin{minipage}[h]{0.3\linewidth}
\center{\includegraphics[width=1\linewidth]{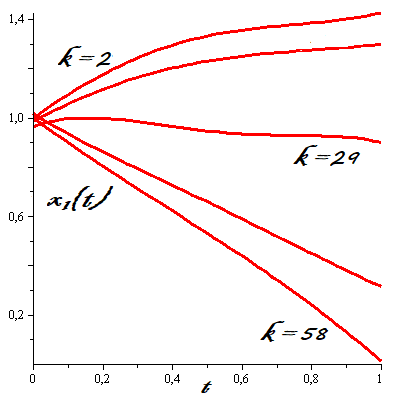} }
\end{minipage}
\hfill
\begin{minipage}[h]{0.3\linewidth}
\center{\includegraphics[width=1\linewidth]{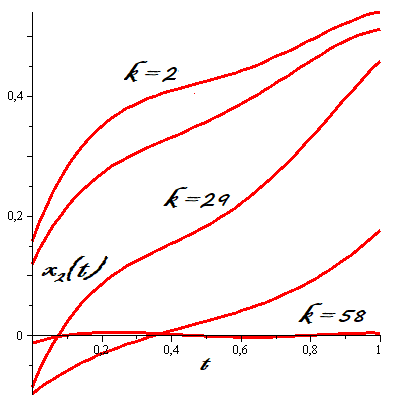} }
\end{minipage}
\hfill
\begin{minipage}[h]{0.3\linewidth}
\center{\includegraphics[width=1\linewidth]{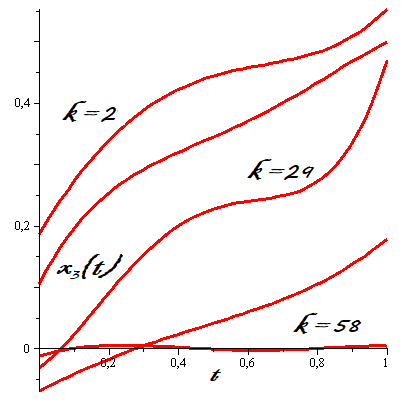} }
\end{minipage}
\caption{Example \ref{ex_4.3}, trajectories on iterations: 2, 13, 29, 45, 58}
\end{figure*}

\end{example}
\begin{example} \label{ex_4.4}
Consider the system
$$\dot x_1(t) = x_2(t)$$
$$\dot x_2(t) = u(t) - P x_2(t) |x_2(t)| - Q x_2(t).$$

It is required to find such a control $u^* \in U$ which brings this system from the initial point $x(0) = (0, 0)'$ to the final state $x(48) = (200, 0)'$ at the moment $T=48$. Herewith, put $\underline u = -2/3$, $\overline u = 2/3$, i. e. we suppose that $-2/3 \leqslant u(t) \leqslant 2/3$ \, $\forall t \in [0, 48]$. The parameters of the problem are $P = 0.78 \times 10^{-4}$ and $Q = 0.28 \times 10^{-3}$. This problem has a practical application to the optimal train motion and was considered in work \cite{outrata}. In fact, in paper \cite{outrata} a more complicated problem is considered with the functional
$$\mathcal J(x, u) = \int_0^{48} x_2(t) \max\big\{u(t), 0  \big\} dt$$
to be minimized. We try to solve optimal control problem (with this functional) via the approach of the research (see Remark \ref{rm5.5.2}).

The problem given is reduced to an unconstrained minimization of the functional 
$$
I(z,u) = \int_0^{48} z_1(t) \max\big\{u(t), 0  \big\} dt + \lambda_1 \int_0^{48} \Big| z_2 (t) - u(t) + P z_1(t) |z_1(t)| + Q z_1(t) \Big | dt + $$
$$ + \lambda_2 \Big( \int_0^{48} z_1(t) dt - 200\Big)^2 + \lambda_3 \Big( \int_0^{48} z_2(t) dt\Big)^2 + $$
$$ +  \lambda_4 \left( \int_0^{48} \max \Big\{-2/3 - u(t), 0 \Big\} dt + \int_0^{48} \max \Big\{{u}(t) - 2/3, 0 \Big\} dt \right) + $$ 
$$ + \lambda_5 \int_0^{48} \Big( z_1(t) - \int_0^t z_2(\tau) d \tau \Big)^2 dt. $$

The functional is slightly simplified beforehand using the fact that $x_2(t) = z_1(t)$, also put $\displaystyle{x_1(t) = \int_0^t z_1(\tau) d \tau}$, $\displaystyle{x_3(t) = \int_0^t z_3(\tau) d\tau }$, $t \in [0, 48]$, throughout iterations. 
 
    

Take $(z_{(1)}, u_{(1)}) = (0, 0, 0, 0)'$ as an initial point, then $I(z_{(1)}, u_{(1)}) = 2 \times 10^5$. As the iteration number increased, the discretization rank gradually increased during the solution
of the auxiliary problem of finding the direction of the quasidifferential descent described in the algorithm and, in the end, the discretization step was equal to $10^{-1}$ (i. e. $N = 480$). The values of penalty factors~$\lambda_i$, $i = \overline{1, 5}$ gradually increased as well from $(5, 5, 5, 5, 5)$ to $(10, 40, 320, 10, 640)$ (in practice, one just needs to monitor whether the restrictions are met for the selected values of these parameters, and, if necessary, to increase them; however, such and approach somewhat reduces the effectiveness of the method, since some actions are assumed to be carried out ``manually'', so in order to automate the process, it is recommended to use some adaptive (automatic) rules for setting the penalty parameter, discussed, for example, in papers \cite{Byrd}, \cite{Mayne2}). \linebreak At the $4569$-th iteration the control  $u_{(4569)}$ was constructed (see the picture) with the value of the functional $I(z_{(4569)}, u_{(4569)}) \approx 12.49101$, herewith $J(x_{(4569)}, z_{(4569)}, u_{(4569)}) \approx 12.48611$, $x_1(T) \approx 199.99793$, $x_2(T) \approx 0.00289$.


After carrying out these iterations the functional value practically stopped decreasing so it was decided to terminate the process at this point. Herewith, $||G(z, u)||_{L^2_2[0,T] \times L^2_1[0,T]} = 0.3$.

Take $u_{(4569)}$ as an approximation to the control $u^*$ sought. In order to verify the result obtained and to find the ``true'' trajectory, we substitute this control into the system given and integrate it numerically. As a result, we have the corresponding trajectory (which is an approximation to the one $x^*$ sought) with the values \linebreak $x_1(T) \approx 199.99607$, $x_2(T) \approx 0.00288$, so we see that the error on the right endpoint does not exceed the value $3 \times 10^{-3}$; the error on the control does not exceed the value $7 \times 10^{-5}$ at each $t \in [0, T]$. The corresponding value of the functional is $\mathcal J(x, u) = 12.48832$.  Due to the minimized functional structure we can not guarantee that the functional value obtained on the $4569$-th iteration is a global minimum in this problem.

The computational time was 38 min 5 sec. The pictures illustrate the resulting control and trajectories obtained via the algorithm realization.

 \begin{figure*}[h!]
\begin{minipage}[h]{0.3\linewidth}
\center{\includegraphics[width=1\linewidth]{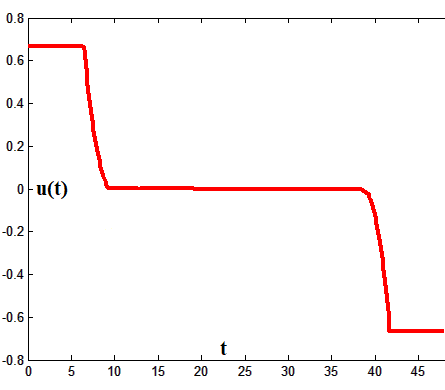} }
\end{minipage}
\hfill
\begin{minipage}[h]{0.3\linewidth}
\center{\includegraphics[width=1\linewidth]{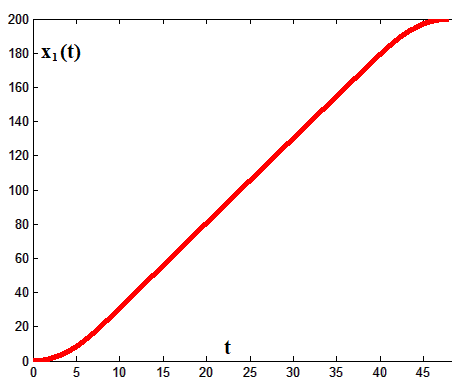} }
\end{minipage}
\hfill
\begin{minipage}[h]{0.3\linewidth}
\center{\includegraphics[width=1\linewidth]{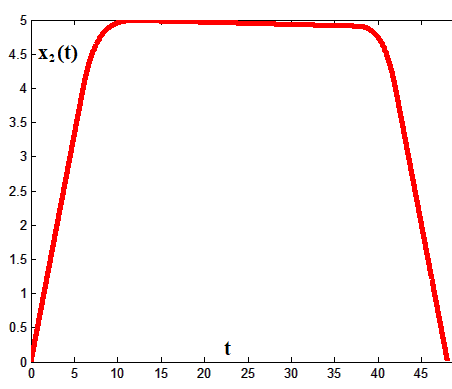} }
\end{minipage}
\caption{Example \ref{ex_4.4}, the resulting control and trajectories}
\end{figure*}
\end{example}
\begin{remark}
Note that in the examples considered, if one suppposes the error on the right endpoint of order $3 \times 10^{-2}$ -- $5 \times 10^{-2}$ to be satisfactory, then the computational time may be reduced at at least two times. This is explained by the fact that the method rather rapidly obtains the localization of solution but after that a lot of iterations may be required to improve the result; what is customary for gradient-type optimization methods. 
\end{remark}
\bigskip
\bigskip
\bigskip

Let us give the solution of this example obtained via an effective DC method \cite{strekal}, \cite{DolgopolikDC} applied to the corresponding finite dimensional problem after direct discretization. Under direct discretization we mean the Euler scheme applied to the system of differential equations and direct left integral Riemann sums substituting the corresponding integrals. Herewith, the discretized functional $\mathcal J(x, u)$ is the objective function and there are constraints in the form of difference scheme equations, restrictions on the right endpoint (the left endpoint is taken from known values) and control. The discretization step value $0.1$ is taken. With such a rank of discretization we obtain the DC optimization problem of dimension $1438$ with~$439$ DC equality constraints. The linear interpolation was used to obtain the corresponding control and trajectory from the values obtained at discrete time moments (see the picture). 
\begin{wrapfigure}{r}{0.5\textwidth} \begin{center}
     \includegraphics[width=0.48\textwidth]{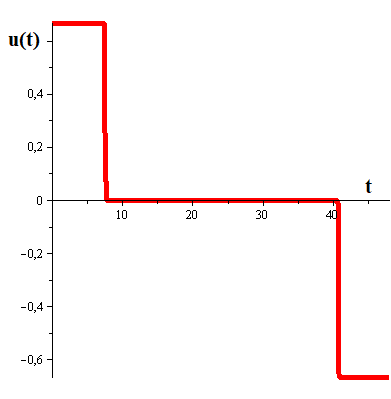}
  \end{center}    
      \end{wrapfigure}
The control obtained delivers the integral functional $\mathcal J(x,u)$ the value $12.59657$. So we see that this result obtained is comparable with the result obtained via present paper approach. However, the error on the right endpoint is rather noticeable (the values \linebreak $x_1(T) = 201.55618$ and $x_2(T) = 0.06521$ are obtained). Note that as one can check the constraints of the discretized problem are satisfied with the accuracy of order $10^{-8}$ on the control obtained, so it is unclear how the restrictions on the right endpoint may be improved to the appropriate values under the difference scheme chosen (note, however, that another difference scheme could certainly significantly improve the accuracy). Apparently, in order to achieve that, one has to take a smaller discretization step. However, note that the further increasing of the discretization rank leads to the dimension of the order $10^4$ (or more), and it seems that (even convex) problems of such a dimension lead to difficulties in ``standard'' machine calculations in general case (unless such problems have a special structure, etc.).
\newline
\newline

Note that if piecewise constant (instead of piecewise linear) interpolation of controls and state derivatives is implemented in DC algorithm, then the result may be significantly improved (the objective function value obtained is approximately $12.43006$ with the error on the right endpoint is $3 \times 10^{-3}$). On the other hand, the method of the paper can be easily modified in order to generate piecewise controls and state derivatives (and continuous piecewise linear trajectories) as well; so it would be correct to compare piecewise constant politics interpolation in both methods. Also note, that in many problems it is natural to obtain continuous controls and phase derivatives from physical considerations. Like this, in the example considered it is natural to suppose that the train speed can not change immediately. 

Note that DC algorithms require a d. c. decomposition \cite{strekal} of the functions; although in this example it is not difficult to obtain such a decomposition, in general case this is a drawback of DC methods unlike the method developed. Note that due to the arbitrary choice of an element of a superdifferential during DC algorithm realization \cite{strekal}, a disadvantage arises when this algorithm may stuck in so called critical points (which are not even stationary) of a functional; although in this example DC method did not stuck in such a critical point, in general case this is a drawback of DC methods unlike the method developed. On the other hand note that DC algorithms are rather effective and may succesfully solve many practical problems and be rather fast. 

\section{Discussion}
First of all, let us briefly explain why our novel idea of the variables $x$ and $z$ ``separation'' is crucial. If this trick is not implemented, on some iteration $k$ of the functional $\mathcal I(z, u)$ minimization algorithm one has to solve the following problem: \begin{equation} \label{d1} \max_{w \in \overline{\partial} \mathcal I(z_{(k)}, u_{(k)}) } \min_{v \in \underline{\partial} \mathcal I(z_{(k)}, u_{(k)}) } \int_0^T \big(v(t) + w(t)\big)^2 dt. \end{equation} However, if one calculates the functional  $\mathcal I(z_{(k)}, u_{(k)})$ quasidifferential than it is seen that the integrand of functional in expression (\ref{d1}) contains, in general case, the functions of the form $\displaystyle \int_0^t {V}(\tau) d \tau$, $\displaystyle \int_0^t {W}(\tau) d \tau$, $t \in [0, T]$. These are Aumann integrals, because $V(\tau)$, $W(\tau)$ belong to some compact sets at each $\tau \in [0, t]$ (and other conditions required of the Aumann integral definition are satisfied as well). It is unclear how to choose the functions $V(t)$, $W(t)$ in this case in order to solve problem (\ref{d1}). The idea implemented allows to get rid of such Aumann integrals in the quasidifferential structure and to solve problem $$\max_{w \in \overline \partial \varphi(\xi_{(k)}, t)} \min_{v \in \underline \partial \varphi(\xi_{(k)}, t)} \big( v(t) + w(t) \big)^2 $$ at each point $t \in [0, T]$ (see (\ref{5.20}) and justification therein). Note that by applying discretization and considering individual selectors of the Aumann integrals at discrete moments of time, problem (\ref{d1}) can be reduced to a finite-dimensional maxmin problem as well, however, the (initial) presence of the indicated Aumann integrals does not allow one to proceed to pointwise calculations, so the finite-dimensional problem obtained in this way becomes extremely computationally expensive (practically unrealizable in the general case even for a small dimension of the original problem). It is necessary to fix the variable under the maximum and to solve the problem for the minimum for each such fixed variable. In this case, for each discretization moment $t_i$, one can take any vector $w(t_i)$ from the set determined by the superdifferential $\overline{\partial} \mathcal I(z_{(k)}, u_{(k)})$ at the time moment $t_i$. For any reasonable rank of discretization, the number of options for fixing the variable under the maximum in such a way (and, accordingly, the number of problems for the minimum obtained in this way) becomes boundless.

The theoretical advantage of the method proposed is as follows: it is original as it is qualitatively different from existing methods based on the direct discretization of the initial problem. Besides, the method preserves attractive geometrical interpretation of quasidifferentials (see book \cite{demrub} for more examples with geometrical illustration in the finite-dimensional case). 

The method proposed has the following practical advantages. The following four paragraphs give examples of some specific problems demonstrating these advantages. In order to simplify the presentation and just to get essence we give examples of some problems of calculus of variations and an example of one simplest control problem only in a smooth case.

Consider the problem of minimizing the functional
$$J(x, z) = \int_0^1 z^4(t)/48 + z^2(t) + x^2(t) - 6x(t) \, dt$$
under the constraints
$$x(0) = 1, \quad x(1) = 0, \quad \int_0^1 x(t) \, dt = 2/3.$$ 
In this example the steepest descent method appeared to be very effective. On the contrary: in order to construct approximations by the Ritz-Galerkin method, a lot of calculations are required and it is also necessary to solve essentially non-linear systems with parameters. 

In a problem of minimizing the 
the functional
$$J(z) = \int_0^{10} z^2(t)- x^2(t) \, dt$$
under the constraints
$$x(0) = 0, \quad x(10) = 0$$ 
both the Euler equation and the Ritz-Galerkin method give a trajectory delivering neither a strong nor a weak minimum. The steepest descent method ``points'' to the fact that there is no solution in this problem: the one-dimensional minimization problem has no bounded solution. 

Minimize the functional
$$J(z) = \int_0^2 z^3(t) \, dt$$
under the constraints
$$x(0) = 0, \quad x(2) = 4.$$ 
This example illustrates that the steepest descent method ``points'' to the fact that, on a solution obtained, the functional reaches a weak minimum rather than a strong one, while both the Euler equation and the Ritz-Galerkin method give only a trajectory delivering a weak minimum.

All the details have been omitted for brevity. One can find a detailed description of these problems and more interesting examples as well as justification of the statements posed in the original papers \cite{demtam}, \cite{tamimm}. Note also that a method used in these papers is slightly different from the one presented but it preserves its many properties, so the comparative analysis is correct. 

\begin{wrapfigure}{r}{0.5\textwidth} \begin{center}
      \includegraphics[width=0.4\textwidth]{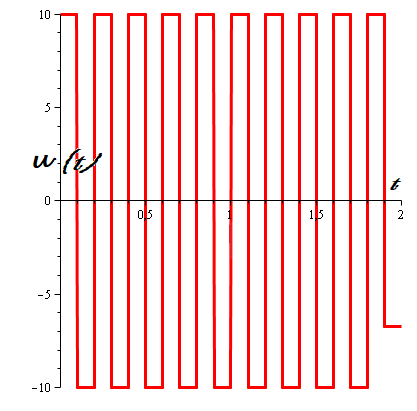}
  \end{center} 
  \end{wrapfigure}

Consider the system
$$\dot x(t) = -x(t) + u(t)$$
on the time interval $[0, 2]$. It is required to find a control $u^* \in P_1[0,T]$ such that the corresponding trajectory satisfies the boundary conditions
$$x(0) = x(2) = 0.$$
\bigskip

Apply direct discretization to this system via the formula
$x(i+1) - x(i) = \frac{1}{N} (-x(i) + u(i))$,
where $i = \overline{0,19}$ and the discretization rank $N=20$.
If we use the initial condition $x_0 = x(0) = 0$ and calculate $x_{20} = x(2)$ as an explicit function of the variables $u_i$, $i = \overline{0,19}$, we will obtain \small $$x_{20}(u) \approx 0.01886768013 \, u_0 + 0.01986071592 \, u_1 + 0.02090601676 \, u_2 + \, 0.0220063334 \, u_3 +$$ $$+0.02316456151 \, u_4 + 0.02438374896 \, u_5 + 0.02566710416 \, u_6 + 0.02701800438 \, u_7 + $$
$$+ 0.02844000461 \, u_8 + 0.02993684696 \, u_9 + 0.03151247049 \, u_{10} + 0.03317102156 \, u_{11} +$$ $$+ 0.03491686480 \, u_{12} + 0.03675459453 \, u_{13} + 0.03868904688 \, u_{14} + 0.04072531250 \, u_{15} +$$ $$+ 0.04286875000 \, u_{16} + 0.045125 \, u_{17} + 0.0475 \, u_{18} + 0.05 \, u_{19}.$$ 
\normalsize So in order to get the required finite position, one has to solve the equation $x_{20}(u) = 0$ with respect to the variables $u_i$, $i = \overline{0,19}$. In other words, it is required to minimize the functional $|x_{20}(u)|$. Take the initial point $u_{(0)}$ with the following coordinates: $u_{2i}=10$, $i=\overline{0,9}$, $u_{2i+1}=-10$, $i=\overline{0,8}$, $u_{19} \approx -6.7101842175$ (see the picture). Note that $x_{20}(u_{(0)}) = 0$, so the point $u_{(0)}$ delivers a global minimum to the functional $|x_{20}(u)|$ (i.~e. the point $u_{(0)}$ solves the discretized problem). However, if we substitute this control $u_{(0)}$ into the original system, we will get the corresponding trajectory $x_{(0)}(t)$ with the finite value $x(2) \approx -0.1189349683$.

Now try to solve this problem via the method of the paper, i. e. minimize the functional
$$I(x,z,u) = \frac{1}{2}\int_0^2 \left(z(t)+x(t)-u(t)\right)^2 dt + \frac{1}{2} \int_0^2 \left( x(t) - \int_0^t z(\tau) d\tau \right)^2 dt + \frac{1}{2} \left(\int_0^2 z(t) dt \right)^2$$
(for simplicity of presentation we have taken a square-function instead of the abs-function as the integrand in the first summand). Take the same point $(x_{(0)}, z_{(0)}, u_{(0)})'$ (where $z_{(0)} = \dot x_{(0)}$) as an initial approximation. One can check that on the first iteration we will get such a control that the corresponding trajectory takes the finite value $|x(2)| < 0.11894$ (i. e. ``better'' than that one obtained via discretization method). 

In fact, applying the method to this example one can get a solution with any given accuracy. This is due the fact that the Gateaux gradient in this case is as follows:
 $$\nabla I(x, z, u) = \begin{pmatrix}
\big(z(t)+x(t)-u(t)\big) +\displaystyle{ \left(x(t) - \int_0^t z(\tau) d \tau \right)} \\
\big(z(t)+x(t)-u(t)\big) - \displaystyle{\left(\int_t^2 \Big( x(\tau) - \int_0^\tau z(s) ds \Big) d\tau \right)} + \int_0^2 z(t) \, dt \\
-\big(z(t)+x(t)-u(t)\big)
\end{pmatrix}.
$$   
Hence, the stopping criteria (that is $||\nabla I(x^*, z^*, u^*)||_{L^2_1[0,2] \times L^2_1[0,2] \times L^2_1[0,2]}  = 0$ in this case) may be only fulfilled (with some accuracy) when the third component of the gradient vanishes. This fact implies that the second summand in the first component vanishes. This fact finally implies that the third summand in the second component vanishes. Thus, we see that there are no local minima of the functional considered, so our method will lead to the desired solution (with any given accuracy). Roughly speaking, the method proposed ``analyzes'' the ``behavior'' of the whole trajectory, rather than only its points considered at some discrete moments of time.   

Although in the example the initial control is chosen in a special way, one can check there is a ``huge'' number of controls with the same properties (i. e. delivering a global minimum to the functional $|x_T(u)|$ but giving an error to the right endpoint). Both increasing the time moment $T$ and adding a nonsmoothness into the right-hand side of the system (e. g. taking the function $-|x(t)|$ instead of $-x(t)$ now) will only increase this ``huge'' number. Of course, while increasing discretization rank, one can decrease the error on the right endpoint. On the other hand, even the discretization rank taken seems to be not very high for control problems solved via discretization or a ``control parametrization'' technique (see, e. g. \cite{teogoh}). Besides, for any discretization rank taken the initial control may be chosen in such a way that it will deliver the global minimum to the functional minimized (i. e. $x_T(u_{(0)}) = 0$) but give an arbitrary large error to the right endpoint.  

As noted, only the smooth case in the examples given is considered for simplicity. A nonsmooth case may lead to even more difficulties: for example, in the control problem presented any kind of nondifferentiability in the right-hand side will significantly increase the number of ``local'' minima with direct discretization used. 

Also note that although some particular concrete examples were given in this section, they demonstrate the general disadvantages of the methods known, so the ``difficulties'' with a lot of nontrivial calculations in Ritz-Galerkin method, ``uninformativeness'' of Euler equations, sufficiently large errors on the right endpoint in discrete methods, etc. may be met in many practical problems.   

List also some secondary advantages of the method proposed:

1) although large discretization rank gives a good approximation to the original problem, the choice of the appropriate discretization rank is not straightforward;

2) in many cases the quasidifferential descent method rapidly demonstrates the structure of the desired solution, although then the convergence to this solution may be very slow;

3) any integral restriction (for example, $||u(t)||^2_{L^2_m [0, T]} \leqslant C$, $C \in \mathbb R$ given) will generate a complicated constraint with a large number of variables (equal to the discretization rank) after direct discretization is applied to the problem; on the contrary: the integral restriction is very natural for the variational statement of the problem solved and it is easy to add a corresponding summand to the functional $I(x, z, u)$ in order to take this restriction into account.

4) if from physical or other considerations, only continuous (in contrast to piecewise continuous) controls and state derivatives are preferable, then the method developed is expected to give better results than the discrete ones (as the discrete method solves the difference scheme system and ``doesn't mind'' the ``behavior'' of the variables between the points of discretization).

5) the algorithm is constructed in such a way that instead of a one problem of large dimension of order $N (n+m)$ (obtained, e. g., from direct discretization of the initial problem) with the discretization rank $N$ one has to solve $N$ problems of dimension of order $2n+m$ of the initial problem which seems more preferable from the computational point of view; the descent directions $G((x_{(k)}, z_{(k)}, u_{(k)}), t_i)$ on the $k$-th iteration are calculated independently for each time moment $t_i$ of discretization, $i = \overline{1, N}$, hence the parallel calculations may be implemented.
\bigskip  

The main disadvantage of the method presented reduces to the computational effort: the number of iterations may be very large. On the other hand, the execution time per one iteration is rather short, so the total time of the algorithm seems satisfactory. 
\bigskip

 \begin{remark}
As noted above, when choosing a step value according to the steepest descent method rule, the convergence of the quasidifferential descent method (MQD) cannot be guaranteed. However, known circumstances allow us to say that ``as a rule'' this method leads to a stationary point of the minimized functional. See also \linebreak Remark \ref{convergence}. 
\end{remark}

\begin{remark} \label{rm_100} 
The most significant advantage of the applied idea of ``separation'' of variables (which has both theoretical and practical implications) is the minimum conditions obtained in pointwise form. 

The theoretical significance lies in the fact that the currently known minimum conditions, for example, in the works of A. D. Ioffe, B. Sh. Mordukhovich, \linebreak E. S. Polovinkin, F. H. Clarke and others (in a similar form for different problems) contain an unknown (``conjugate'') function, which is a solution to a special differential inclusion. Checking these conditions even for a given trajectory (and the control that generates it) presents significant difficulties and can not always be carried out (since due to the difficulties while calculating the corresponding subdifferential in the general case it is unclear how to look for the conjugate vector-function, which is a solution to a complex differential inclusion, even with using sampling). The verification of the minimum conditions obtained in the paper is universal in the sense that it can be carried out for the specific control and the corresponding trajectory (at separate times of discretization), albeit for an ``approximate'' problem (at large values of the penalty parameters). In this regard, note the minimum conditions obtained in the works of \linebreak M. V. Dolgopolik in terms of quasidifferentials, which are also ``algorithmically'' verifiable in the general case.

The practical significance lies in the possibility of constructing a numerical method. (We note that when constructing numerical methods based on the above-mentioned known minimum conditions, in the general case of quasidifferentiable functions fundamental difficulties arise.) Moreover, in the proposed numerical method the descent directions in the algorithm can be searched independently of each other at individual moments of discretization. Thus, instead of one problem of large dimension $N(n+m)$ (obtained if we apply direct discretization), one needs to solve $N$ problems of smaller dimension $2n+m$, where $N$ is the discretization rank, and $n$ is the dimension of the phase vector and $m$ is the dimension of control. This ``replacement'' seems promising from a computational point of view; in addition, due to the mutual independence noted of these problems, parallel calculations can be used.
\end{remark}

\begin{remark} \label{convergence} Let us make a remark regarding the quasidifferential descent (MQD) method, which is often used to solve various problems of the paper. In the case of choosing a step according to the steepest descent method rule, the convergence of the MQD cannot be guaranteed. Even in the finite-dimensional case, there are examples when the implementation of this method causes a "jamming" effect, which leads to convergence to a nonstationary point of the minimized function. However, the structure of such functions in the constructed examples \cite{demmal} is very specific. This specificity consists, in particular, in the fact that if the descent step according to the rule given is calculated not exactly, but with an arbitrarily small error, and also if at some point the functions are considered active also with some error (which is what happens in practice), then the method will converge to a stationary point (see the corresponding modifications and justification of the convergence of the method for a certain class of finite-dimensional problems in \cite{demmal} and \cite{demgamsil} and in works \cite{Fominyhmin}, \cite{Fominyhmax}). In practice, the MQD has shown itself positively (from the point of view of convergence) and has been successfully applied to solving specific problems. The circumstances described allow us to say that ``as a rule'' this method leads to a stationary point of the minimized functional. The proof of the ``weak'' convergence of the MQD and the modification of the MQD for a special class of functionals, namely, those containing a minimum or maximum of continuously differentiable functions under the integral, is given in works \cite{Fominyhmin}, \cite{Fominyhmax}. \end{remark}

\section{Conclusion}
The paper is devoted to developing a direct ``continuous'' method for a nonsmooth control problem. The problem of bringing a system with a nondifferentiable (but only quasidifferentiable) right-hand side from one point to another is considered. The admissible controls are those from the space of piecewise continuous vector-functions which belong to some parallelepiped at each moment of time. The proposed approach can be applied to nonsmooth optimal control problem in Lagrange form (additionally the integral with a quasidifferentiable integrand is to be minimized). The problem of finding the steepest (the quasidifferential) descent direction was solved and the quasidifferential descent method was applied to some illustrative examples. The method is original and is qualitatively different from the existing methods since the majority of them are based on direct discretization of the original problem. The main and new idea implemented is to consider phase trajectory and its derivative as independent variables and to take the natural relation between these variables into account via penalty function of a special form. This idea gives possibility 1) to obtain qualitatively new optimality conditions in nonsmooth optimal control problem in pointwise form (which may be effectively checked at discrete time moments), 2) to calculate the quasidifferential of the minimized functional and eventually to construct the steepest descent direction method for solving this problem. 

\section* {Acknowledgements}
The author is sincerely greatful to his colleagues Maksim Dolgopolik and Grigoriy Tamasyan for numerous fruitful discussions.


\begin{thebibliography}{}

\bibitem[Balakrishnan(1974)]{balakrishnan}
{\it Balakrishnan A.} Introduction to Optimization Theory in Hilbert Space. M.: Mir, 1974.
260 p.

\bibitem[Berezin \& Zhidkov(1962)]{BerezinZhidkov}
{\it Berezin I. S., Zhidkov N. P.} Calculation methods. Vol. 1. Ed. 2. M.: Fizmatlit, 1962. 464 p.

\bibitem[Byrd \& Nocedal \& Waltz(2008)]{Byrd} {\it Byrd R. H., Nocedal J., Waltz R. A.} Steering exact penalty methods for nonlinear
programming // Optimization Methods and Software. 2008. Vol. 23, no. 2. P. 197--213.

\bibitem[Dolgopolik(2023)]{DolgopolikDC} {\it Dolgopolik M. V.} Steering exact penalty DCA for nonsmooth DC optimisation problems with
equality and inequality constraints // Optimization Methods and Software. 2023. Vol. 38,
iss. 4. P. 668--697.

\bibitem[Demyanov \& Rubinov(1977)]{demrub}
{\it Demyanov V. F., Rubinov A. M.} Basics of nonsmooth analysis and quasidifferential calculus. 1990.
Moscow: Nauka. (in Russian)

\bibitem[Demyanov \& Dolgopolik(2013)] {dolgopolik1} {\it Demyanov V. F., Dolgopolik M. V.} Codifferentiable functions in Banach spaces: methods and applications to problems of variation calculus // Vestnik of St. Petersburg University. \linebreak Applied Mathematics. Computer Science. Control Processes. 2013. V. 3. P. 48--66. \linebreak (in Russian)

\bibitem[Vinter \& Cheng(1998)]{Vinter1} {\it Vinter R. B., Cheng H.} Necessary conditions for optimal control problems with state constraints // Transactions of the American Mathematical Society. 1998. V. 350. no. 3. \linebreak P. 1181--1204.

\bibitem[Vinter(2005)]{Vinter2}{\it Vinter R. B.} Minimax optimal control // SIAM J. Control Optim. 2005. V. 44. no. 3. \linebreak P. 939--968.

\bibitem[Frankowska(1984)]{Frankowska} {\it Frankowska H.} The first order necessary conditions for nonsmooth variational and control problems // SIAM J. Control Optim. 1984. V. 22. no. 1. P. 1--12.

\bibitem[Mordukhovich(1989)]{Mordukhovich} {\it Mordukhovich B.} Necessary conditions for optimality in nonsmooth control problems with nonfixed time // Differential Equations. 1989. V. 25. no. 1. P. 290--299.

\bibitem[Ioffe(1984)]{Ioffe}  {\it Ioffe A. D.} Necessary conditions in nonsmooth optimization // Mathematics of Operations Research. 1984. V. 9. no. 2. P. 159--189. 

\bibitem[Shvartsman(2007)]{Shvartsman}  {\it Shvartsman I. A.} New approximation method in the proof of the Maximum Principle for nonsmooth optimal control problems with state constraints // Journal of Mathematical Analysis and Applications. 2007. V. 326. no. 2. P. 974--1000.

\bibitem[Ito \& Kunisch(2011)]{ito}  {\it Ito K., Kunisch K.} Karush---Kuhn---Tucker conditions for nonsmooth mathematical programming problems in function spaces // SIAM J. Control Optim. 2011. V. 49. no. 5. \linebreak P. 2133--2154.

\bibitem[De Oliveira \& Silva(2013)] {olivera} {\it De Oliveira V. A., Silva G. N.} New optimality conditions for nonsmooth control problems // Journal of Global Optimization. 2013. V.~57. no. 4. P.~1465--1484.


\bibitem[Fominyh(2019)] {Fominyh2}  {\it Fominyh A. V.} Open-Loop control of a plant described by a system with nonsmooth right-hand side // Computational Mathematics and Mathematical Physics. 2019. V. 59. Iss. 10. P. 1639--1648.

\bibitem[Fominyh(2024)] {Fominyhmin}  {\it Fominyh A. V.} Method for Finding Solution to Nonsmooth Differential Inclusion of \linebreak Special
Structure // ESAIM: Control, Optimisation and Calculus of Variations, 2024. \linebreak Vol. 30.
 P. 1--26.

\bibitem[Fominyh(2025)] {Fominyhmax}  {\it Fominyh A. V.} Method for finding solution to ``quasidifferentiable'' differential inclusion // arXiv:2406.15384.

\bibitem[Demyanov, Nikulina \&  Shablinskaya(1986)]{demshablnik}  {\it Demyanov V. F., Nikulina V. N., Shablinskaya I. R.} Quasidifferentiable functions in optimal control // Mathematical Programming Study. 1986. V. 29. P. 160--175.

\bibitem [Fominyh, Karelin \&  Polyakova(2018)]  {Fominyh1}{\it Fominyh A. V., Karelin V. V., Polyakova L. N.} Application of the hypodifferential descent method to the problem of constructing an optimal control // Optimization Letters. 2018. V. 12. no. 8. P. 1825--1839.

\bibitem [Fominyh(2017)] {Fominyh3}{\it Fominyh A. V.} Methods of subdifferential and hypodifferential descent in the problem of constructing an integrally constrained program control // Automation and Remote Control. 2017. V. 78. P. 608--617.

\bibitem [Fominyh(2021)] {Fominyh4}{\it Fominyh A. V.} The quasidifferential descent method in a control problem with nonsmooth objective functional // Optimization Letters. 2021. V. 15. no. 8. P. 2773--2792.

\bibitem [Outrata(1983)] {outrata} {\it Outrata J. V.} On a class of nonsmooth optimal control problems // Applied Mathematics \& Optimization. 1983. V. 10. no. 1. P. 287--306.

\bibitem[Gorelik \& Tarakanov(1992)] {Gorelik1}  {\it Gorelik V. A., Tarakanov A. F.} Penalty method and maximum principle for nonsmooth variable-structure control problems // Cybernetics and Systems Analysis. 1992. V. 28. Iss.~3. P. 432--437. 

\bibitem[Gorelik \& Tarakanov(1989)] {Gorelik2}  {\it Gorelik V. A., Tarakanov A. F.} Penalty method for nonsmooth minimax control problems with interdependent variables // Cybernetics. 1989. V. 25. Iss. 4. P. 483--488. 

\bibitem[Morzhin(2009)] {Morzhin}  {\it Morzhin O. V.} On approximation of the subdifferential of the nonsmooth penalty functional in the problems of optimal control // Avtomatika i Telemekhanika. 2009. no. 5. P. 24--34. (in Russian)

\bibitem[Mayne \& Polak(1985)] {Mayne1}  {\it Mayne D. Q., Polak E.} An exact penalty function algorithm for control problems with state and control constraints // 24th IEEE Conference on Decision and Control. 1985. \linebreak P. 1447--1452.

\bibitem[Mayne \& Smith(1988)] {Mayne2}  {\it Mayne D. Q., Smith S.} Exact penalty algorithm for optimal control problems
with control and terminal constraints // International Journal of Control. 1988. V. 48. no. 1. P. 257--271.

\bibitem[Noori Skandari, Kamyad \& Effati(2015)] {Skandari1}  {\it Noori Skandari M. H., Kamyad A. V., Effati S.} Smoothing approach for a class of \linebreak nonsmooth optimal control problems // Applied Mathematical Modelling. 2015. V. 40. no. 2. \linebreak P. 886--903.

\bibitem[Noori Skandari, Kamyad \& Effati(2013)] {Skandari2}  {\it Noori Skandari M. H., Kamyad A. V., Erfanian H. R.} Control of a class of nonsmooth \linebreak dynamical systems // Journal of Vibration and Control 2013. V. 21. no. 11. P. 2212--2222.

\bibitem[Ross \& Fahroo(2004)] {RossFahroo} {\it Ross I. M., Fahroo F.} Pseudospectral knotting methods
for solving optimal control problems // Journal of Guidance, Control, and Dynamics. 2004. V.~27. no. 3. P. 397--405.

\bibitem[Demyanov \& Tamasyan(2011)] {demtam} {\it Demyanov V. F., Tamasyan, G. Sh.} Exact penalty functions in isoperimetric problems // Optimization. 2011. V. 60. no. 1.  P. 153--177.

\bibitem[Demyanov \& Vasil'ev(1986)] {demvas}  {\it Demyanov V. F., Vasil'ev L. V.} Nondifferentiable optimization. 1986. New York: Springer-Optimization Software.

\bibitem[Kolmogorov \& Fomin(1999)] {kolfom}{\it Kolmogorov A. N., Fomin S. V.} Elements of the theory of functions and functional analysis. 1999. New York: Dover Publications Inc.

\bibitem[Dolgopolik(2011)] {DolgNormed}{\it Dolgopolik M. V.} Codifferential calculus in normed spaces // Journal of Mathematical Sciences. 2011. V. 173. no. 5. P. 441--462.

\bibitem[Blagodatskikh \& Filippov(1986)] {filblag} {\it Blagodatskikh V. I., Filippov A. F.} Differential inclusions and optimal control // Proc. Steklov Inst. Math. 1986. Vol. 169. P. 199--259.

\bibitem[Aubin \& Frankowska(1990)] {AubinFrankowska} {\it Aubin J.-P., Frankowska H.} Set-valued analysis. 1990. Boston: Birkhauser Basel.


\bibitem[Filippov(1959)] {Filippov} {\it Filippov A. F.} On certain questions in the theory of optimal control // Journal of the Society for Industrial and Applied Mathematics, Series A: Control. 1959. V. 1. Iss. 1. P. 76--84.

\bibitem[Munroe(1953)] {Munroe} {\it Munroe M. E.} Introduction to measure and integration. 1953. Massachusetts: Addison-Wesley.

\bibitem[Dunford \& Schwartz(1958)] {DunfordSchwarz} {\it Dunford N., Schwartz J. T.} Linear Operators, Part 1: General Theory. 1958. New York: Interscience Publishers Inc.
\newline 

\bibitem[Dolgopolik(2018)] {DolgConverg} {\it Dolgopolik M. V.} A convergence analysis of the method of codifferential descent // Computational Optimization and Applications. 2018. V. 71. P. 879--913.

\bibitem[Dolgopolik(2014)] {Dolgopolikcodiff} {\it Dolgopolik M. V.} Nonsmooth problems of calculus of variations via codifferentiation // ESAIM: Control, Optimisation and Calculus of Variations. 2014. V. 20. no. 4.
P. 1153--1180. 

\bibitem[Vasil'ev(2002)] {Vasil'ev} {\it Vasil'ev F. P.} Optimization methods. 2002. Moscow: Factorial Press. (in Russian) 


 \bibitem[Wolfe(1959)] {wolfe} {\it Wolfe P.}
 The simplex method for quadratic programming // Econom. 1959. V.~27. P.~382--398.

\bibitem[Ryaben'kii(2008)]{ryab} {\it Ryaben'kii V. S.} Introduction to computational mathematics. 2008. Moscow: Fizmatlit. \linebreak (in Russian)

\bibitem[Cullum(1969)] {Cullum1969} {\it Cullum J.} Discrete approximations to continuous optimal control problems // SIAM J. \linebreak Control. 1969. V. 7. no. 1. P. 32--49.



\bibitem[Demyanov \& Malozemov(1990)]{demmal} {\it Demyanov F. F., Malozemov V. N.} Introduction to minimax. 1990. New York: Dover Publications Inc.

\bibitem[Dolgopolik \& Fominyh(2019)]{dolgfom} {\it Dolgopolik M. V., Fominyh A. V.} Exact penalty functions for optimal control problems I: Main theorem and free-endpoint problems // Optimal Control Applications and Methods. 2019. V. 40. Iss. 6. P. 1018--1044.

 \bibitem[Krylov(1968)]{krylov} {\it Krylov I. A.} Numerical solution of the problem of the optimal stabilization of an artificial satellite // USSR Computational Mathematics and Mathematical Physics. 1968. V. 8. Iss.~1. P. 284--291.
 
 \bibitem[Strekalovsky(2020)]{strekal} {\it Strekalovsky A. S.} Local search for nonsmooth DC optimization with DC equality and
inequality constraints / In A. M. Bagirov, M. Gaudioso, N. Karmitsa, M. M. Makela,
and S. Taheri, editors. Numerical Nonsmooth Optimization. State of the Art Algorithms.
\linebreak P. 229--262. Springer, Cham, 2020.
 
 \bibitem[Demyanov \& Tamasyan(2010)] {tamimm} {\it Demyanov V. F., Tamasyan, G. Sh.} On direct methods for solving variational problems // Trudy Inst. Mat. i Mekh. UrO RAN. 2010. V. 16. no. 5. P. 36--47. (in Russian)
 
\bibitem[Demyanov \& Gamidov \& Sivelina(2010)] {demgamsil} {\it Demyanov V. F., Gamidov S., Sivelina, T. I.} An algorithm for minimizing a certain class of
quasidifferentiable functions // Quasidifferential Calculus, 1986. P. 74--84.
 
\bibitem[Teo \& Goh  \& Wong(1991)]{teogoh} {\it  Teo K. L., Goh C. J., Wong K. H.} A unified computational approach to optimal control problems. 1991. New York: Longman Scientific and Technical.





%
%
%
%
%
%
%
%
%
%
%
%
%
%
%
%
%
%
%
%

\end{thebibliography}
\end{document}